\documentclass[microtype]{gtpart}     
\agtart
\usepackage[usenames,dvipsnames]{color}
\usepackage[all]{xypic}
\usepackage{makeidx}
\usepackage{pinlabel}
\usepackage{graphicx}

\usepackage{calc}

\usepackage{stmaryrd}

\title{The non-multiplicativity of the signature modulo $8$ of a fibre bundle is an Arf--Kervaire invariant}

\author{Carmen Rovi}
\givenname{Carmen}
\surname{Rovi}

\address{Department of Mathematics,
Indiana University,
Bloomington, Indiana 47405}
\email{crovi@indiana.edu}
\urladdr{http://pages.iu.edu/~crovi/index.html}

 \title[The signature modulo $8$ of a fibration]{The non-multiplicativity of the signature modulo 8 of a fibre bundle is an Arf--Kervaire invariant}

\arxivreference{}
\arxivpassword{}

\volumenumber{}
\issuenumber{}
\publicationyear{}
\papernumber{}
\startpage{}
\endpage{}
\doi{}
\MR{}
\Zbl{}
\received{}
\revised{}
\accepted{}
\published{}
\publishedonline{}
\proposed{}
\seconded{}
\corresponding{}
\editor{}
\version{}

\newtheorem{thm}{Theorem}[section]
\newtheorem{theorem}[thm]{Theorem}

\newtheorem{lemma}[thm]{Lemma}

\newtheorem{rem}{Remark}

\theoremstyle{definition}
\newtheorem{definition}[thm]{Definition}

\newtheorem{example}[thm]{Example}

\newtheorem{rmk}[thm]{Remark}

\newtheorem{proposition}[thm]{Proposition}

\newcommand{\bbD}{\mathbb{D}}

\newcommand{\zz}{\mathbb Z}

\newcommand{\bb}[1]{\mathbb{#1}}
\newcommand{\mc}[1]{\mathcal{#1}}

\newcommand{\cP}{\mathcal P}

\newcommand{\pe}{\pi_1(E)}
\newcommand{\bbDA}{\mathbb{D}(\mathbb{A})}

\newcommand{\Fd}{F^{\textnormal{dual}}}

\newcommand{\wtE}{\widetilde{E}}

\newcommand{\wtF}{\widetilde{F}}

\newcommand{\ZZ}{\mathbb Z}

\usepackage{hyperref} 
\usepackage{array} 
\usepackage{paralist} 
\usepackage{verbatim} 
\usepackage{subfig} 

\begin{document}

\begin{abstract}
It was proved by Chern, Hirzebruch and Serre that the signature of a fibre bundle $F\to E \to B$ is multiplicative if the fundamental group $\pi_1(B )$ acts trivially on $H^*(F; \bb{R})$, with $\sigma(E)=\sigma(F)\sigma(B )$. Hambleton, Korzeniewski and Ranicki
proved that in any case the signature is multiplicative modulo $4$, $\sigma(E)=\sigma(F)\sigma(B)\bmod 4$. In this paper we present two results concerning the multiplicativity modulo $8$: firstly we identify $(\sigma(E)-\sigma(F)\sigma(B))/4 \bmod 2$ with a $\zz_2$-valued Arf-Kervaire invariant of a Pontryagin squaring operation. Furthermore, we prove that if $F$ is $2m$-dimensional and the action of $\pi_1(B)$ is trivial on $H^m(F, \bb{Z})/torsion \otimes \bb{Z}_4$, this Arf-Kervaire invariant takes value $0$ and hence the signature is multiplicative modulo $8$, $\sigma(E)=\sigma(F) \sigma(B ) \bmod 8$.

\end{abstract}

\maketitle


\section*{Introduction} \label{Introduction}

The signature of a nondegenerate symmetric bilinear form on a finite dimensional real vector space is the number of positive definite summands minus the number of negative definite summands in a splitting of the form into $1$-dimensional nondegenerate bilinear forms. The signature $\sigma(X)\in \zz$ of an oriented $4k$-dimensional Poincar\'e duality space $X$ with fundamental class $[X] \in H_{4k}(X)$ is the signature of the symmetric bilinear form $(a, b) \mapsto \langle a \cup b, [X] \rangle$.


Chern, Hirzebruch and Serre proved in \cite{HirzebruchSerreChern}  that multiplicativity of the signature of a fibre bundle holds when the action of the fundamental group $\pi_1(B)$ on the cohomology ring of the fibre $H^*(F; \bb{R})$ is trivial,
$$\sigma(E) - \sigma(B)\sigma(F) =0 \in \bb{Z}.$$
Later on Kodaira, Hirzebruch and Atiyah constructed non-multiplicative fibre bundles, with the action of $\pi_1(B)$ action on $H^*(F;\bb{R})$ necessarily nontrivial.

The signature of a fibre bundle is multiplicative modulo 4, whatever the action. This was proved by Meyer in  \cite{Meyerpaper} for surface bundles and by Hambleton, Korzeniewski and Ranicki in  \cite{modfour} for high dimensions. The two main results in this paper are the following. In Theorem \ref{4Arf-general-fibration} we identify the obstruction to multiplicativity of the signature modulo $8$ of a fibre bundle with a $\zz_2$-valued \textit{Arf-Kervaire invariant}. Moreover we shall prove that if the action of $\pi_1(B)$ is trivial on $H^m(F, \bb{Z})/torsion \otimes \bb{Z}_4$, this Arf invariant takes value $0$. That is, we shall prove the following theorem:

{\bf Theorem \ref{mod8-theorem}} \qua
{Let $F^{2m} \to E^{4k} \to B^{2n}$ be an oriented Poincar\'e duality fibration. If the action of $\pi_1(B)$ on $H^m(F, \bb{Z})/torsion \otimes \bb{Z}_4$ is trivial, then
$$\sigma(E)  -\sigma(F) \sigma(B) =0\in \zz_8.$$}
A key feature for our study of the signature of a fibration is obtaining a model for the chain complex of the total space which gives us enough information to compute its signature. It has been known since the work by Chern, Hirzebruch and Serre \cite{HirzebruchSerreChern} that the signature of the total space depends only on the action of the fundamental group of the base $\pi_1(B)$ on the cohomology of the fibres.
Clearly it is not possible to construct the chain complex of the total space by taking into account only the action of $\pi_1(B).$ For example, the base space of the Hopf fibration
$S^1 \to S^3 \to S^2 $
has trivial fundamental group $\pi_1(S^2)= \{ 1 \}$, but the chain complex of the total space in this case is not a product.
So taking into account the information from the chain complexes of the base and fibre and the action of $\pi_1(B)$ is not enough to construct the chain complex of the total space, but it is enough to construct a model that will detect the signature.

The model that we will develop here is inspired by the transfer map in quadratic $L$-theory constructed in \cite{SurTransfer} by L\"uck and Ranicki.
In \cite{SurTransfer} the surgery transfer of a fibration $F^m \to E \to B^n$ with fibre of dimension $m$ and base of dimension $n$ is given by a homomorphism
$$p^{!} \co L_n(\bb{Z}[\pi_1(B)]) \longrightarrow  L_{n+m}(\bb{Z}[\pi_1(E)]). $$

L\"uck and Ranicki also prove in \cite{SurTransfer} that the surgery transfer map in quadratic $L$-theory agrees with the geometrically defined transfer maps. A similar transfer map does not exist in symmetric $L$-theory.
There are two obstructions to lifting a symmetric chain complex $(C, \phi) \in L^n(\bb{Z}[\pi_1(B)])$ to an $(m+n)$-dimensional chain complex $p^!(C, \phi) \in L^{n+m}(\bb{Z}[\pi_1(E)])$ which are described in the appendix of \cite{SurTransfer}.
Basically the difference with the quadratic $L$-theory transfer lies in the fact that the symmetric $L$-groups are not $4$-periodic in general, so that one cannot assume that surgery below the middle dimension can be performed to make $(C, \phi)$  highly-connected.

The chain model that we shall discuss in this paper will provide a well defined map
$$L^n(\bb{Z}[\pi_1(B)]) \longrightarrow  L^{n+m}(\bb{Z}).$$
The foundations for the construction of this map where laid by L\"uck and Ranicki in \cite{SurTransfer} and \cite{SurObstructions}. We follow the description  of the chain model by Korzeniewski in \cite[chapters 3, 4]{Korzen}.
The construction of the model for the total space uses the fact that the chain complex of the total space is \textit{filtered}, so the idea for this construction is similar to that of the Serre spectral sequence. This was the approach taken by Meyer in \cite{Meyerpaper} where he describes the intersection form of the total space of a surface bundle in terms of the intersection form on the base with coefficients in a local coefficient system.

\subsection*{Thanks} This paper comprises results from my PhD thesis \cite{mythesis} and I would like to thank my supervisor Prof Andrew Ranicki for suggesting the topic and for his constant support.  Moreover I would like to thank the School of Mathematics of the University of Edinburgh for providing an excellent research environment and the Max Planck Institute for Mathematics in Bonn for their hospitality during the time this paper was written. I would also like to thank the anonymous referee for many helpful comments and suggestions.

\section{The algebraic model for the signature of a fibration} \label{model}

The ideas and notation in the diagram in figure \ref{overview} have not yet been introduced. The purpose of the diagram is to give an overview of the key steps in the construction of a suitable model for a fibration $F \to E \to B$, which we will explain in the first section of this paper.

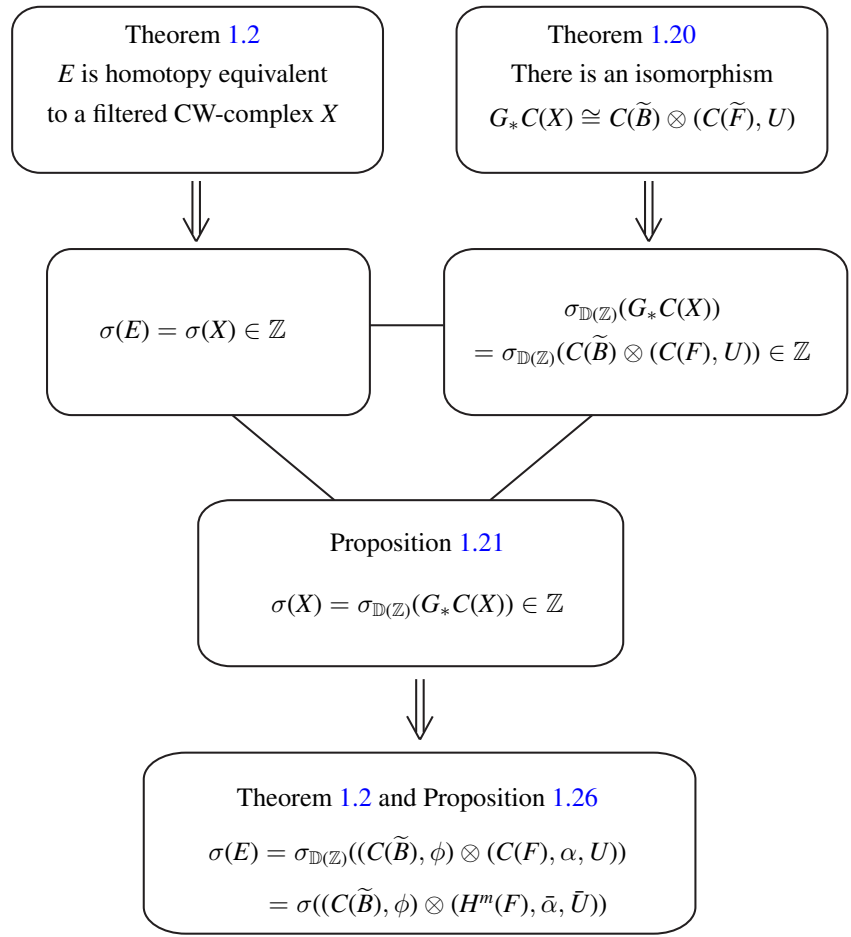
\begin{figure} \label{overview}
\labellist
\small\hair 2pt
\pinlabel ${\textnormal{Theorem \ref{E and X}}}$ at 130 445
\pinlabel ${E \textnormal{ is homotopy equivalent} }$ at 130 425
\pinlabel ${\textnormal{to a filtered CW-complex $X$} }$ at 130 405
\pinlabel ${\textnormal{Theorem \ref{chain-iso}}}$ at 340 445
\pinlabel ${\textnormal{There is an isomorphism} }$ at 350 425
\pinlabel ${G_*C(X) \cong C(\widetilde{B}) \otimes (C(\widetilde{F}), U) }$ at 350 405
\pinlabel ${\sigma(E) = \sigma (X) \in \bb{Z}}$ at 130 300
\pinlabel ${\sigma_{\bb{D}(\bb{Z})}(G_*C(X) ) }$ at 350 310
\pinlabel ${= \sigma_{\bb{D}(\bb{Z})}(C(\widetilde{B}) \otimes (C(F), U)) \in \bb{Z}}$ at 350 290
\pinlabel ${\textnormal{Proposition \ref{L-derived}}}$ at 240 195
\pinlabel ${\sigma (X) = \sigma_{\bb{D}(\bb{Z})}(G_*C(X) )  \in \bb{Z}}$ at 240 165
\pinlabel ${\textnormal{Theorem \ref{E and X} and Proposition \ref{two-functors}}}$ at 240 70
\pinlabel ${\sigma (E) = \sigma_{\bb{D}(\bb{Z})} ((C(\widetilde{B}), \phi) \otimes (C(F), \alpha, U))}$ at 240 45
\pinlabel ${ =\sigma ((C(\widetilde{B}), \phi) \otimes (H^m(F), \bar{\alpha}, \bar{ U}))}$ at 250 20
\endlabellist
\centering
\includegraphics[scale=0.77]{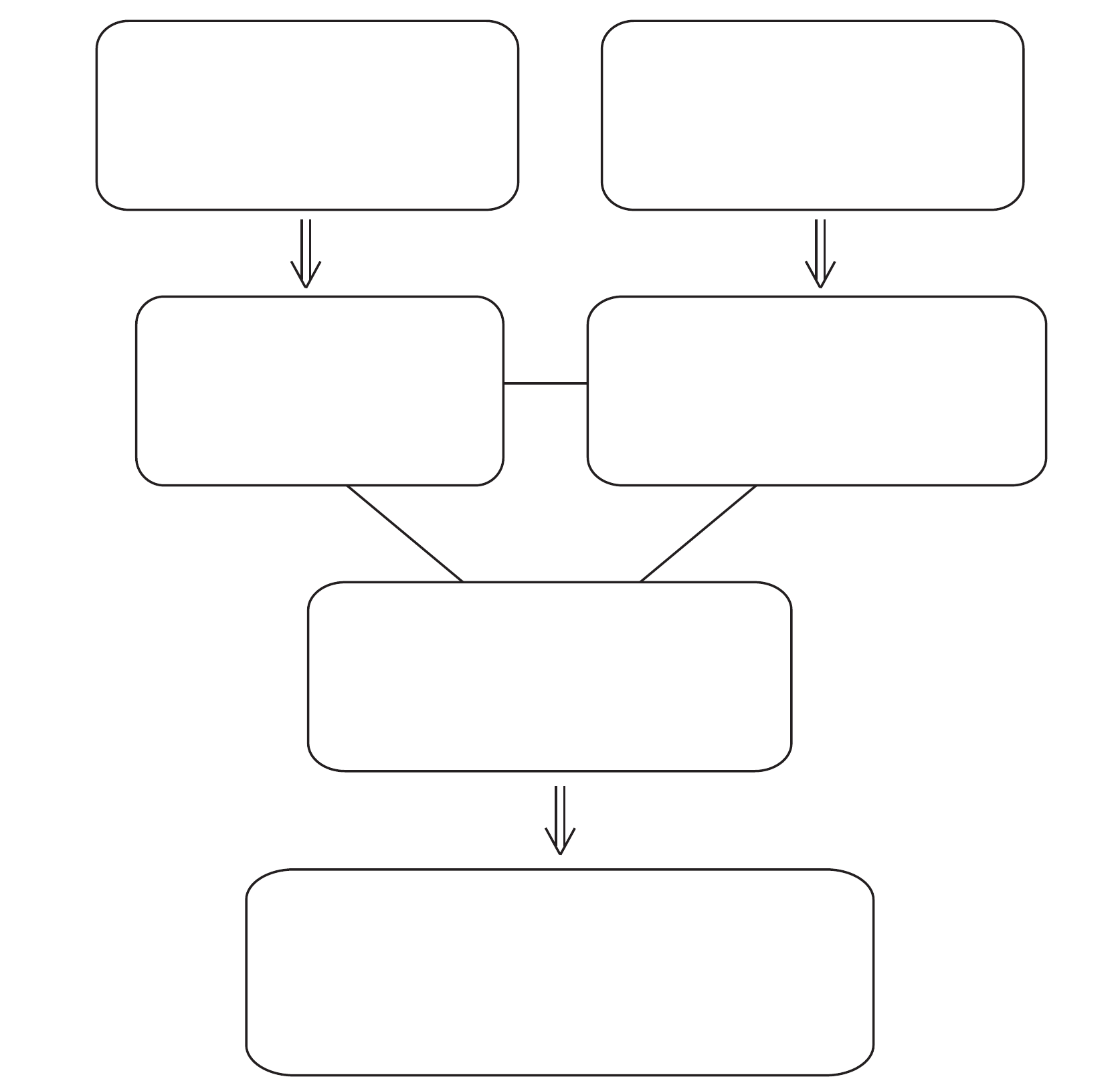}
\caption{Overview of the construction for the chain complex model of the total space.}
\label{overview}
\end{figure}

\subsection{Fibrations and $\Gamma$-fibrations} \label{fibrations}


We will consider Hurewicz fibrations where the base space $B$ will be a path connected CW-complex and for any point $b_0 \in B$, the fibre $F= p^{-1}(b_0)$ has the homotopy type of a finite CW-complex. Furthermore all the fibres have the same homotopy type.

In order to consider the most general possible setting for the construction of our chain complex algebraic model of a fibration we will need to use the definition of $\Gamma$-fibration from L\"uck \cite{Transfer-K}.

\begin{definition}{ \rm (L\"uck \cite[Definition 1.1]{Transfer-K})} \qua
Let $\Gamma$ be a discrete group. A \textbf{$\Gamma$-fibration} is a $\Gamma$-equivariant map $p' \co E' \to B'$ with $E'$ a $\Gamma$-space  such that $\Gamma$ acts trivially on $B'$ and $p'$ has the $\Gamma$-equivariant lifting property for any $\Gamma$-space $X$
\begin{displaymath}
\xymatrix{X \ar[r]^{f} \ar[d]_{x \mapsto (x, 0)} & E' \ar[d]^{p'} \\
X \times I \ar[ur]^{H'} \ar[r]^{F} & B'.}
\end{displaymath}
\end{definition}

We will consider a fibration $p\co E \to B$ with fibre $F = p^{-1}(b)$. The map $q \co \widetilde{E}\to E$ will denote the universal cover of $E$. The composition $p \circ q = \widetilde{p} \co \widetilde{E} \to B$ is a $\Gamma$-fibration with $\Gamma= \pe$. This fibration has fibre $\widetilde{p}^{-1}(b)=\widetilde{F}$, which is the cover of $F$ induced from the universal cover of $E$
\begin{displaymath}
\xymatrix{\widetilde{F} \ar[r] \ar[d] & \widetilde{E} \ar[d]^{q} \ar[dr]^{\widetilde{p}= p \circ q} \\
F \ar[r] & E \ar[r]^p & B.}
\end{displaymath}

\subsection{Filtrations}\label{Filtrations}

The argument to construct an algebraic model of the total space appropriate for the computation of the signature is motivated by the well known result that the total space of a fibration is filtered. The following Theorem \ref{E and X} was proved by Stasheff in \cite{Stasheff} using an inductive argument and by Schoen in \cite{Schoen} using a different argument based on the CW approximation theorem of Whitehead \cite{Whitehead}, see also Spanier \cite[p. 412]{Spanier}.
A similar argument to that of Stasheff \cite{Stasheff} was used by L\"uck in \cite{Transfer-K}, by Hambleton, Korzeniewski and Ranicki in \cite{modfour} and by Korzeniewski in  \cite{Korzen}.

\begin{theorem}{\rm (Stasheff and Schoen \cite{Stasheff, Schoen}) } \qua \label{E and X}
Let $F \to E \to B$ be a Hurewicz fibration where $B$ and $F$ have the homotopy type of CW-complexes, then the total space $E$ is weakly homotopy equivalent to a CW-complex $X$.
\end{theorem}

\subsubsection{Filtered spaces} \label{Filtered spaces}
Here we shall define what is meant by filtered spaces. We only consider compactly generated spaces, such as CW complexes.

\begin{definition}\label{522}
A \textbf{$k$-filtered topological space} $X$ is a topological space which is equipped with a series of subspaces
$$X_{-1}= \emptyset \subset X_0 \subset X_1 \subset \dots \subset X_k=X.$$
Here we will assume that each of the inclusions $X_j \subset X_{j+1}$ is a \textit{cofibration}.
\end{definition}
The condition that  $X_j \subset X_{j+1}$ is a  cofibration implies that the pair $(X_{j+1}, X_j)$ has the \textit{homotopy extension property.} (See Hatcher \cite[page 14]{Hatcher}). In particular, the inclusion of a CW subcomplex is a cofibration, and has the homotopy extension property.

\begin{definition} Let $X$ and $Y$ be two filtered spaces.
\begin{itemize}
\item[(i)] A map $f \co X \to Y$ is a \textbf{filtered map} if $f(X_j) \subset Y_j$.
\item[(ii)] A \textbf{filtered homotopy} between the maps $f \co X \to Y$ and $g \co X \to Y$ is a homotopy $H \co X \times I \to Y$ such that $H(X_j \times I) \subset Y_j$.
\item[(iii)] Two filtered spaces $X$ and $Y$ are {\bf filtered homotopy equivalent} if there exist filtered maps $f \co X \to Y$ and $h \co Y \to X$ such that $hf$ is filtered homotopic to $\mathrm{Id}_X$ and $fh$ is filtered homotopy equivalent to $\mathrm{Id}_Y$. In such a case the filtered map $f$ will be called a {\bf filtered homotopy equivalence} between $X$ and $Y$.
\end{itemize}
\end{definition}

In the context established in this section, namely that each of the inclusions $X_j \subset X_{j+1}$ (or $Y_j \subset Y_{j+1}$) is a cofibration, the following lemma holds.
\begin{lemma} A filtered map
 $f \co X \to Y$ is a filtered homotopy equivalence if and only if each $f_j\co X_j \to Y_j$  is a homotopy equivalence of unfiltered spaces.
\end{lemma}

\begin{proof} See Brown \cite[7.4.1]{Brown-book}.
\end{proof}

\begin{definition}
A \textbf{$k$-filtered CW-complex $X$} is a CW-complex $X$ together with a series of subcomplexes
$$X_{-1}=\emptyset \subset X_0 \subset X_1 \subset \dots \subset X_k=X.$$
The cellular chain complex $C(X)$ is filtered with
$$F_jC(X) = C(X_j).$$
\end{definition}

A $k$-filtered CW-complex satisfies the conditions of Definition \ref{522}.

The main application will be a filtered complex in the context of Theorem \ref{E and X}, a fibration $E \xrightarrow{p} B$ with $B$ a CW-complex. We will consider $B$ with a filtration given by its skeleta which induces a filtered structure on $E$ by defining $E_k := p^{-1}(B_k),$ where $B_k$ is the $k$-th skeleton of $B$. Note that here the inclusions $E_{k-1} \subset E_k$ are cofibrations.

\subsubsection{Filtered complexes} \label{Filtered-cx}

Here $\bb{A}$ will denote an additive category.
\begin{definition}\label{FM}
\begin{itemize}
\item[(i)] Let $M$ be an object in the additive category $\bb{A}$ and let $M$ have a direct sum decomposition
$$M = M_0 \oplus M_1 \oplus \dots \oplus M_k,$$
so that $M$ has a filtration of length $k$
\begin{gather*}
F_{-1}M=0 \subseteq F_0M \subseteq F_1M \subseteq \dots \subseteq F_kM = M \\
\tag*{\text{where}} F_iM = M_0 \oplus M_1 \oplus \dots \oplus M_i.
\end{gather*}
A \textbf{$k$-filtered object $F_*M \in \bb{A}$} is the object $M \in \bb{A}$ together with the direct sum decomposition of $M$.

\item[(ii)]
Let $F_*M$ and $F_*N$ be two $k$-filtered objects in  the additive category $\bb{A}$.
A \textbf{filtered morphism} is given by
$$
f = \left( \begin{array}{ccccc} f_0 & f_1 & f_2 & \dots & f_k \\
                                              0  & f_0 & f_1 & \dots & f_{k-1} \\
                                              0 &    0  & f_0 & \dots & f_{k-2} \\
                                              \vdots & \vdots & \ddots & \vdots& \vdots \\
                                               0 & 0 & 0 & \dots & f_0 \end{array} \right) \co M = \bigoplus^{k}_{s=0} M_s \to N= \bigoplus^{k}_{s=0} N_s.
$$
\end{itemize}

\end{definition}

In the context of chain complexes in the additive category $\bb{A}$, a $k$-filtered complex $F_*C$ is defined as follows.

\begin{definition}\label{filtered-chain-complex}
Let $C \co C_n \to \dots \to C_r \to C_{r-1} \to \dots  \to C_{0}$ be a chain complex, and let each $C_r$ be $k$-filtered, that is, $$C_r = C_{r, 0} \oplus C_{r, 1} \oplus \dots \oplus C_{r, s} \oplus \dots \oplus C_{r, k} = \bigoplus^k_{s=0} C_{r,s}.$$
Then a \textbf{$k$-filtered complex $F_*C$} in $\bb{A}$ is a finite chain complex $C$ in $\bb{A}$ where each of the chain groups is $k$-filtered and each of the differentials $d\co F_*C_r \xrightarrow{d} F_*C_{r-1}$ is a filtered morphism. The matrix components of $d$ are the maps $C_{r,s} \xrightarrow{d_j} C_{r-1, s-j}.$

\end{definition}

$F_sC_r = \sum\limits^s_{i=0}C_{r,i}$, and $F_sC_r/F_{s-1}C_r = C_{r,s}$. So $C_{r,s}$ represents the $s$-th filtration quotient of $C_r$.

Tensor products of filtered complexes are carefully described by Hambleton, Korzeniewski and Ranicki in \cite[section 12.2]{modfour}.
\begin{definition} The \textbf{tensor product of filtered chain complexes} $F_*C$ and $F_*D$ over $\bb{A}(R)$ and $\bb{A}(S)$, where $R$ and $S$ are rings, is itself a filtered complex
$$F_k(C \otimes_{\bb{Z}} D) = \bigoplus_{i+j = k} F_iC \otimes_{\bb{Z}} F_jD.$$
\end{definition}

\subsection{The associated complex of a filtered complex} \label{derived section}

\subsubsection{A complex in the derived category $\bb{D}(\bb{A})$} \label{complex-derived}

Given an additive category $\bb{A}$ we write $\bb{D}(\bb{A})$ for the homotopy category of $\bb{A}$, which is the additive category of finite chain complexes in $\bb{A}$ and chain homotopy classes of chain maps with
$$\textnormal{Hom}_{\bb{D(A)}}(C, D) = H_0(\textnormal{Hom}_\bb{A} (C, D)).$$
See L\"uck and Ranicki \cite[Definition 1.5]{SurTransfer}.

 A $k$-filtered complex $F_*C $  in $\bb{A}$ has an associated chain complex in the derived category $\bbDA$. The associated complex of a $k$-filtered space is denoted by $G_*(C)$  and is $k$-dimensional
\begin{displaymath}
G_*(C)\co G_k (C)\to \dots \to G_r(C) \xrightarrow{d_*} G_{r-1}(C) \to \dots  \to G_{0}(C).
\end{displaymath}
The \textit{morphisms} in $G_*(C)$ are given by the (filtered) differentials
$$d_*=(-)^s d_1 \co G_k(C)_s = C_{k+s, r}\to G_k(C)_{s-1}=C_{k+s-1, r-1}.$$
Each of the individual terms $G_r(C)$ is an \textit{object} in $\bbDA$, hence a chain complex in $\bb{A}$
$$G_r(C) \co \dots \to G_r(C)_s \to G_r(C)_{s-1} \to G_r(C)_{s-2} \to \dots .$$
As a chain complex, $G_r(C)$ has differentials
\begin{gather*}d_{G_r(C)} = d_0 \co G_r(C)_s = C_{r+s, r} \to G_r(C)_{s-1}= C_{r+s-1, r} \\
\tag*{\text{such that}}
(d_{G_r(C)})^2 = d_0^2 =0 \co G_r(C)_s= C_{r+s, r} \to G_r(C)_{s-2}=C_{r+s-2, r}. \end{gather*}
Note that the differentials of a filtered complex $d\co C_k \to C_{k-1}$ are such that  $d^2=0 \co C_r \to C_{r-2}$. These differentials are upper triangular matrices. If we write $d_0$ for the diagonal, $d_1$ for the superdiagonal we obtain some relations:
\begin{gather*}
d_0^2 = 0 \co C_{r, s} \to C_{r-2, s} \\
d_0d_1 + d_1 d_0=0 \co C_{r,s} \to C_{r-2, s-1} \\
(d_1)^2 + d_0 d_2 + d_2d_0 =0 \co C_{r,s} \to C_{r-2, s-2},
\end{gather*}
up to sign.
These are the required relations for the objects in the associated complex to be in $\bbDA$ and for the differentials to be morphisms in $\bbDA$ with square $0$.

\begin{example}
For a filtered CW-complex $X$
$$G_kC(X) = S^{-k}C(X_k, X_{k-1}).$$
\end{example}

\subsubsection{Duality for a filtered complex and its associated complex in the derived category} \label{Duality}
Let $\mathbb{A}$ be an additive category with involution. Then $\bbD_n(\mathbb{A})$ is the additive category of $n$-dimensional chain complexes in $\mathbb{A}$ and chain homotopy classes of chain maps, with $n$-duality involution $T \co \bbD_n(\mathbb{A}) \to \bbD_n(\mathbb{A}) ; C \mapsto C^{n-*}.$

The definition of a filtered dual chain complex in an additive category with involution $\bb{A}$ is given by Hambleton, Korzeniewski and Ranicki in \cite[section 12.6]{modfour}. 
Here we review the definition of the dual $\Fd_*C$ of a $k$-filtered chain complex $F_*C$.

\begin{definition}{(Hambleton, Korzeniewski and Ranicki \cite[Definition 12.21 (ii)]{modfour})} \qua
Let $F_*C$ be a $k$-filtered $n$-dimensional chain complex in $\bb{A}$.

\begin{itemize}
\item[(i)] The filtered dual $\Fd_*C$ of $F_*C$ is the $k$-filtered complex with modules $$(\Fd_*C)_{r, s} = C^*_{n-r, k-s},$$
where $0\leq r \leq n$ and $0 \leq s \leq k$.
\item[(ii)] The dual of the differential $C_r \xrightarrow{d} C_{r-1}$ is given by $d^{\textnormal{dual}}\co (\Fd_*C)_r \to (\Fd_*C)_{r-1}$. This dual differential is also $k$-filtered
$$C^*_{n-r, k-s} \xrightarrow{(-)^{r+s +j(n+r)} d^*_j}C^*_{n-(r-1), k-(s-j)}.$$
\end{itemize}
\end{definition}

We refer the reader to Hambleton, Korzeniewski and Ranicki \cite[Section 12.6]{modfour} for more details on filtered complexes and their associated graded complexes. The application in this paper is for the $k$-filtered chain complex $F_*C(E)$ of the total space of a fibre bundle $F \to E \to B^k$ with $\textnormal{dim}~ B= k.$

\begin{lemma}{(Hambleton, Korzeniewski and Ranicki \cite[12.23]{modfour})} \label{duality-graded} \qua Let $F_*C$ be a $k$-filtered $(n+k)$-dimensional chain complex with associated complex $G_*(C)$. The associated complex $G_*(F^{\textnormal{dual}}_*C) = (G_*(C))^*$ is the $(n, k)$-dual of $G_*(C)$.
\end{lemma}

\begin{rmk}
Note that Hambleton, Korzeniewski and Ranicki use the terminology $n$-dual of $G_*(C)$ instead of $(n, k)$-dual. We believe that this latter terminology is more precise since the construction of the dual of $G_*(C)$ uses $n$-duality and in each fixed degree $0, 1,\dots, k$ and then reverses the order in $\{0, 1, \dots, k \}$.
\end{rmk}

\subsection{The transfer functor associated to a fibration}\label{transfer-sec}

Let $F \to E \to B$ be a fibration where $B$ is based and path connected. It is possible to model the fibration by a map of topological monoids $\Omega B \to \textnormal{Map}(F, F),$ where $\Omega B$ is grouplike. Broadly speaking, the fibration can be recovered in the form
$$EB  \times_{\Omega B} F \longrightarrow EB \times \{pt\},$$
where $EB$ is a contractible space with an action of $\Omega B$ which is free. In particular we have $EB \times_{\Omega B} \{pt \} \simeq B.$

\begin{figure}
\labellist
\small\hair 2pt
\pinlabel ${h}$ at 126 117
\pinlabel ${p^{-1}(x) = F}$ at 300 121
\pinlabel ${F}$ at 27 93
\pinlabel ${F \times I}$ at 57 93
\pinlabel ${T(F \xrightarrow{h} F)}$ at 170 93
\pinlabel ${E}$ at 367 93
\pinlabel ${p}$ at 134 61
\pinlabel ${p}$ at 64 61
\pinlabel ${p}$ at 315 61
\pinlabel ${x}$ at 122 41
\pinlabel ${x}$ at 279 38
\pinlabel ${I}$ at 27 32
\pinlabel ${S^1}$ at 107 32
\pinlabel ${\omega}$ at 185 38
\pinlabel ${B}$ at 367 32
\endlabellist
\centering
\includegraphics[scale=0.95]{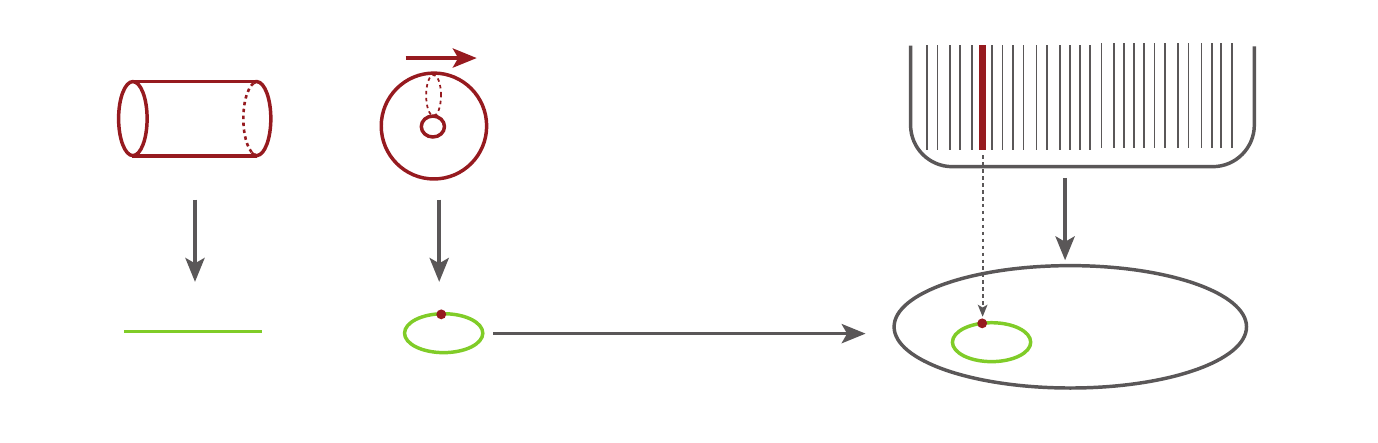}
\caption{Monodromy automorphism}
\label{monodromy}
\end{figure}
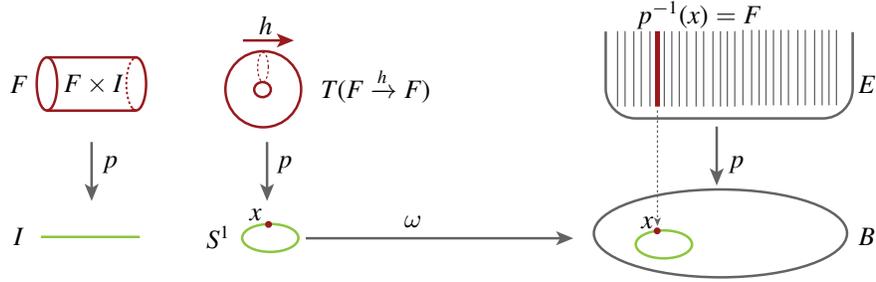

For any loop $\omega \co S^1 \to B$ the pullback $F \to \omega^*E \to S^1 $ is the mapping torus $$\omega^*E~=~T(h\co F \to F)~=~F \times I/\{(y,0) \sim (h(y),1)\,\vert\,y \in F\}$$
of the monodromy automorphism $h= U(\omega)\co F \to F.$ (See Figure \ref{monodromy}).

 We are considering the action of $\Omega B$ on the fibre $F$
$$\Omega B \to \textnormal{Map}(F, F); ~ \omega \mapsto h.$$
If $\omega$ is a loop in $\Omega B$, the homotopy class of its corresponding map $h_{\omega}\co F \to F$ only depends on the homotopy class of $\omega \in \pi_1(B)$. So there is an action of the fundamental group $\pi_1(B)$ on the fibre $F$ given by the fibre transport
$$u\co \pi_1(B) \to [F, F]^{op}.$$
This determines
$$ u \co \pi_1(B) \to [C(F), C(F)]^{op}.$$
which extends to a ring morphism
\begin{equation}\label{ring-morph}
U \co H_0(\Omega B) = \bb{Z}[\pi_1(B)] \to H_0(\textnormal{Hom}_{\bb{Z}}(C(F), C(F)))^{op}
\end{equation}
from $\bb{Z}[\pi_1(B)]$ to the opposite of the ring of chain homotopy classes of $\zz$-module chain maps $h\co C(F) \to C(F)$.
If $F$ is a Poincar\'e space then $U$ is a morphism of rings with involution. 
Let $\alpha$ be the symmetric structure on $C(F).$ Then the involution on $H_0(\textnormal{Hom}_{\mathbb{Z}} (C(F), C(F))^{\textnormal{op}})$ is defined by $T(h) = \alpha_0^{-1} h^* \alpha_0,$
where $\alpha_0 \co C(F)^{n-*} \to C(F)$ is the zeroth part of the symmetric structure $\alpha.$


It is also possible to give a $\pi_1(E)$-equivariant version of the fibre transport by considering the $\pi_1(E)$-fibration  $\wtE \to B$, with fibres $\wtF$ the pullback to $F$ of the universal cover $\wtE$ of $E$.
In this case the ring morphism is given by
\begin{equation}\label{equiv-ring-morph}
U \co H_0(\Omega B) = \bb{Z}[\pi_1(B)] \to H_0(\textnormal{Hom}_{\bb{Z}[\pi_1(E)]}(C(\wtF), C(\wtF)))^{op}.
\end{equation}

The ring morphisms from \eqref{ring-morph} and \eqref{equiv-ring-morph} induce transfer functors as we now explain.

The idea of the transfer functor associated to a fibration  $F \to E \xrightarrow{p} B$ was developed by L\"uck and Ranicki in \cite{SurTransfer}.

\begin{definition}{(L\"uck and Ranicki \cite[Definition 1.1]{SurTransfer})} \qua A \textbf{representation} $(A, U)$ of a ring $R$ in an additive category $\bb{A}$ is an object $A$ in $\bb{A}$ together with a morphism of rings $U \co R \to \textnormal{Hom}_{\bb{A}}(A, A)^{op}.$
\end{definition}

Following the notation used by L\"uck and Ranicki in \cite{SurTransfer} we will denote by $\bb{B}(R)$ the additive category of based finitely generated free $R$-modules, where $R$ is an associative ring with unity. As above, $\bb{D(A)}$ denotes the derived homotopy category of the additive category $\bb{A}$.
From L\"uck and Ranicki \cite[Definition 1.5]{SurTransfer} we know that
$$\textnormal{Hom}_{\bb{D(A)}}(C, D) = H_0(\textnormal{Hom}_{\bb{A}}(C, D)).$$

\begin{definition}{(L\"uck and Ranicki \cite{SurTransfer})} \qua
A representation $(A, U)$ determines a \textbf{transfer functor} $-\otimes (A, U) = F \co \bb{B}(R) \to \bb{D(A)}$ as follows
$$
\begin{array}{l}F(R^n) = A^n, \\
F\left((a_{ij} \co R^n \to R^m) \right) = \left(U(a_{ij})\right) \co A^n \to A^m.
\end{array}
$$
\end{definition}
The ring morphism induced by the fibre transport
$$U\co \bb{Z}[\pi_1(B)] \to H_0(\textnormal{Hom}_{\bb{Z}[\pi_1(E)]} (C(\wtF), C(\wtF)))^{op}$$
determines the representation $(C(\wtF), U)$ of the ring $\bb{Z}[\pi_1(B)]$ into the derived category $\bb{D}(\bb{B}(\bb{Z}[\pi_1(E)]))$ and hence the following functor
$$-\otimes (C(\wtF), U) \co \bb{B}(\bb{Z}[\pi_1(B)]) \to \bb{D}(\bb{B}(\bb{Z}[\pi_1(E)])).$$

Now let $\bb{A}$ denote an additive category with involution as in L\"uck and Ranicki \cite{SurTransfer}.
\begin{definition}\label{sym-rep}{(L\"uck and Ranicki \cite{SurTransfer})}  \qua
A \textbf{symmetric representation} $(A, \alpha, U)$ of a ring with involution $R$ in an additive category with involution $\bb{A}$ is a nonsingular symmetric form $(A, \alpha)$ in $\bb{A}$ together with a morphism of rings with involution $U \co R \to \mathrm{Hom}_{\bb{A}}(A, A)^{op}$.
\end{definition}
For a fibration $F^{2m} \to E \to B^{2n}$ with $F$ Poincar\'e there is defined a symmetric representation $(C(\tilde{B}), \alpha, U)$ in the derived  category $\bb{D}_{2m}(\bb{B}(\bb{Z}[\pi_1(E)]))$, where $U$ is a  ring morphism of rings with involution
$$U\co \bb{Z}[\pi_1(B)] \to H_0(\textnormal{Hom}_{\bb{Z}[\pi_1(E)]} (C(\wtF), C(\wtF)))^{op}.$$
The symmetric representation $(C(\tilde{B}), \alpha, U)$ determines the functor
$$-\otimes (C(\wtF), \alpha, U) \co \bb{B}(\bb{Z}[\pi_1(B)]) \to \bb{D}_{2m}(\bb{B}(\bb{Z}[\pi_1(E)])).$$
This functor induces maps in the $L$-groups.
It is defined and used by L\"uck and Ranicki in \cite{SurTransfer} to construct a transfer map in the quadratic $L$-groups associated to a fibration $F^{2m} \to E^{2n+2m} \to B^{2n}$
$$L_{2n}(\bb{Z}[\pi_1(\widetilde{B})]) \xrightarrow{-\otimes (C(\wtF), \alpha, U)} L_{2n}(\bb{D}_{2m}(\bb{B}(\bb{Z}[\pi_1(E)])) \xrightarrow{\mu} L_{2n+2m}(\bb{Z}[\pi_1(E)]),  $$
where $\mu$ is the Morita map also defined by L\"uck and Ranicki in \cite{SurTransfer}.

Note that $\bb{D}_{2m}(\bb{B}(\bb{Z}[\pi_1(E)]))$ is an additive category with involution. The symmetric $L$-groups of such categories are defined by Ranicki in detail in \cite[Chapter 1]{bluebook}.
In symmetric $L$-theory the transfer functor still induces a map
$$L^{2n}(\bb{Z}[\pi_1(\widetilde{B})]) \xrightarrow{-\otimes (C(\wtF), \alpha, U)} L^{2n}(\bb{D}_{2m}(\bb{B}(\bb{Z}[\pi_1(E)]))),$$
but there is no Morita map, unlike the situation in quadratic $L$-theory.

In our context, since we are only interested in obtaining information about the ordinary signature of $E$, so we can forget about $\pi_1(E)$. That is, we only need to work with the map which we will make precise in the section \ref{Filtered-signature}.
$$L^{2n}(\bb{Z}[\pi_1(B)])\xrightarrow{-\otimes (C(F), \alpha, U)} L^{2n}(\bb{D}_{2m}\bb{Z})  \xrightarrow{\sigma_{\bb{D}_{2m} {\bb{Z}}}} \bb{Z}.$$

\subsection{Symmetric structures}
Let $W$ be the standard free $\bb{Z}[\bb{Z}_2]$ resolution of $\bb{Z}$
$$W\co \dots \to \bb{Z}[\bb{Z}_2] \xrightarrow{1-T} \bb{Z}[\bb{Z}_2] \xrightarrow{1+T}   \bb{Z}[\bb{Z}_2]  \xrightarrow{1-T}   \bb{Z}[\bb{Z}_2] \to 0.$$
Let $C$ be a finite f.g. left $R$-module chain complex, where $R$ is a ring with involution. Then its transpose $R$-module is the right $R$-module denoted by $C^{t}$ with the same additive group.

A \textbf{symmetric structure} of degree $k$ on $C$ is a chain map $W \to C^t \otimes_R C$ raising degrees by $k$ which is equivariant for the actions of $\zz_2$.


\subsubsection{Tensor products of chain complexes in the derived category}

For the definition of the tensor product in a derived  category $\mathbb{D}_n(\mathbb{A})$ we refer to Ranicki \cite[page 27]{bluebook}. This reference discusses the tensor product of objects in an additive category with chain duality.
We shall only need the special case when the additive category is the derived category  $\bbD_n (\mathbb{A})$  with the $n$-duality involution $C \mapsto C^{n-*}$.

\begin{definition}\label{Ran-tensor-prod} For any objects $M$, $N$ in the additive category with involution $\bbD_n (\mathbb{A})$ define the abelian group
$$M \otimes_{\bbD_n (\mathbb{A})} N = \textnormal{Hom}_{\bb{D}_n(A)}(M^*, N),$$
where $M^*$ is the dual of $M$.
\end{definition}

In the application in this paper we will consider a fibration $F^{2m} \to E \to B^{2n}$ with $F$ and $B$ Poincar\'e. We will construct chain complex models for the signature of the total space, which are symmetric chain complexes in $L^{2n}(\mathbb{D}_{m}(\mathbb{Z})).$



A typical representative of an element of $L^{2n}(\bb{D}_{2m} (\ZZ))$ is a chain complex $C$ made up of ojects in $\bb{D}_{2m} (\ZZ)$, i.e., it is a diagram in the category $\bb{D}_{2m} (\ZZ)$ of the shape
$$\bullet \leftarrow \bullet \leftarrow \bullet \leftarrow \dots \leftarrow \bullet \leftarrow \bullet $$
where the composition of any two consecutive arrow is zero, together with a $2n$-cycle
$$ \phi \in \textnormal{Hom}_{\ZZ[\ZZ_2]}(W, C \otimes_{\bb{D}_{2m} (\ZZ)} C),$$
where 
$$C \otimes_{\bb{D}_{2m} (\ZZ)} C:= \textnormal{Hom}_{\bb{D}_{2m} (\ZZ)}(C^{-*}, C).$$




\subsubsection{Symmetric structure induced by the representation associated to a fibration}

In this section we will see that the symmetric representation functor associated to a fibration  $F^{2m} \to E \to B^{2n}$ induces a symmetric structure in the derived additive category with involution $\bb{D}_{2m}(\zz)$. The details of this construction can be found in \cite[Theorems 3.11, 4.5]{Korzen}.

The symmetric representation
\begin{equation}\label{sym-rep} U\co \bb{Z}[\pi_1(B)] \to H_0(\textnormal{Hom}_{\bb{Z}} (C(F), C(F))^{op})\end{equation}
induces a natural morphism
$$\textnormal{Hom}_{\zz[\pi_1(B)]}\left(C(\widetilde{B})^*,  C(\widetilde{B})\right) \xrightarrow{U}\textnormal{Hom}_{\bb{D}_{2m}(\zz)}\left((C(\widetilde{B}) \otimes (C(F), U))^*,  C(\widetilde{B}) \otimes (C(F), U)\right).$$

Let $W$ be the standard free $\bb{Z}[\bb{Z}_2]$ resolution of $\bb{Z}.$
The action of $T \in \zz_2$ on $\textnormal{Hom}_{\zz[\pi_1(B)]}\left(C(\widetilde{B})^*,  C(\widetilde{B})\right)$ is preserved  under the natural transformation $U$, so the induction map preserves the duality,
\begin{equation}\label{chain-map}\begin{split}
\textnormal{Hom}_{\bb{Z}[\bb{Z}_2]}\left(W, \textnormal{Hom}_{\zz[\pi_1(B)]}\left(C(\widetilde{B})^*,  C(\widetilde{B})\right)\right)  \xrightarrow{U} \\ \textnormal{Hom}_{\bb{Z}[\bb{Z}_2]}\left(W, \textnormal{Hom}_{\bb{D}_{2m}(\zz)}\left((C(\widetilde{B}) \otimes (C(F), U))^*,  C(\widetilde{B}) \otimes (C(F), U)\right)\right).
\end{split}
\end{equation}
This induces a map in homology.

\begin{definition}\label{symm-str-induced-by-U}
The $2n$-dimensional derived symmetric structure on the associated chain complex in $\bb{D}_{2m}( \bb{Z})$ of the filtered complex $C(\widetilde{B}) \otimes (C(F), \alpha, U)$ induced by the symmetric representation $(C(F), \alpha, U)$ is the cycle in $$H_{2n}\left(\textnormal{Hom}_{\bb{Z}[\bb{Z}_2]}\left(W, \textnormal{Hom}_{\bb{D}_{2m}(\zz)}\left((C(\widetilde{B}) \otimes (C(F), U))^*,  C(\widetilde{B}) \otimes (C(F), U)\right)\right)\right).$$
obtained by evaluating the chain map in \eqref{chain-map} on the preferred cycle there, namely the symmetric structure on $C(\widetilde{B})$.
\end{definition}

\begin{definition}
The $2n$-dimensional $Q$-group $Q^{2n}(C(\widetilde{B}) \otimes (C(F), U))$ is the abelian group of equivalence classes of  derived $2n$-dimensional symmetric structures on $C(\widetilde{B}) \otimes (C(F),\alpha, U)$,
\begin{equation*}\begin{split}Q^{2n}(C(\widetilde{B}) \otimes (C(F), \alpha, U)) =  \\ H_{2n}\left(\textnormal{Hom}_{\bb{Z}[\bb{Z}_2]}\left(W, \textnormal{Hom}_{\bb{D}_{2m}(\zz)}\left((C(\widetilde{B}) \otimes (C(F), U))^*,  C(\widetilde{B}) \otimes (C(F), U)\right)\right)\right). \end{split}\end{equation*}
\end{definition}

The symmetric  representation $(C(F), \alpha, U)$ induces morphisms of abelian groups which preserve the duality. Therefore it induces a morphism of the $Q$-groups,
$$Q^{2n}(C(\widetilde{B})) \to Q^{2n}(C(\widetilde{B}) \otimes (C(F), \alpha, U)). $$
Hence there is a map in $L$-theory
$$-\otimes (C(F), \alpha, U) \co L^{2n}(\bb{Z}[\pi_1(B)])\rightarrow L^{2n}(\bb{D}_{2m}\bb{Z}),$$
and we can construct a symmetric chain complex $(C(\widetilde{B}), \phi) \otimes  (C(F), \alpha, U)$ over $\bb{D}_{2m}(\zz)$, i.e., a chain complex made up of objects in $\bb{D}_{2m}(\zz)$.

\subsubsection{Symmetric structure on the graded complex associated to the filtered complex $X$}\label{Iso-derived}
In this paper we are considering a fibration $F^{2m} \to E \xrightarrow{p} B^{2n}$ with $F$ and $B$ Poincar\'e.
Since $E \xrightarrow{p} B$ is a fibration, then $E \times E \xrightarrow{p \times p} B \times B$ is also a fibration.
There is a commutative diagram of the diagonal maps  $\Delta^E \co E \to E \times E$ and $\Delta^B \co B \to B \times B$,

\begin{displaymath}
\xymatrix{ 
E \ar[r]^-{\Delta^E}  \ar[d]_p & E \times E \ar[d]^{p \times p} \\
 B \ar[r]^-{\Delta^B} & B \times B.}
\end{displaymath}

From Theorem \ref{E and X} we know that the total space of a fibration is weakly homotopy equivalent to a filtered CW-complex $X$ with the filtration $X_0 \subset X_1 \subset \dots  \subset X_{2n} = X$ induced from the cellular structure of the base space. Up to homotopy we can define the the following composition of maps,
$$
\xymatrix{
X \ar@/^2pc/[rrrr]^{\Delta^X} \ar[r] ^-h & E \ar[r]^-{\Delta^E} & E \times E  \ar[rr]^-{h^{-1} \times h^{-1}}& & X \times X.
}$$


The chain approximation $C(X)\xrightarrow{\Delta_{C(X)}} C(X)\otimes_{\bb{Z}} C(X)$ can be chosen to be natural, see Bredon \cite[Chapter VI. 16]{Bredon}. Naturality means that for every inclusion $X_j \hookrightarrow X_{j+1}$ there is defined a commutative square

\begin{displaymath}
\xymatrix{ 
F_jC(X)_r = C(X_j)_r \ar[r]^-{\Delta_i}  \ar[d] & (C(X_{j}) \otimes_{\bb{Z}} C(X_j))_{r+i} \ar[d] \\
 C(X_{j+1})_r \ar[r]^-{\Delta_i} & (C(X_{j+1}) \otimes_{\bb{Z}} C(X_{j+1}))_{r+i}.}
\end{displaymath}

Thus, by the naturality of the diagonal chain approximation, we can assume that $\Delta_{C(X)}$ preserves the filtration. 


Since $X$ is a filtered complex over $\bb{Z}$, it has (as explained in Section \ref{complex-derived}) an associated complex $G_*C(X) = F_*C(X)/F_{*-1}C(X) $ in $\bb{D}_{2m}(\bb{Z})$, where $\bb{D}_{2m}(\zz)$ is the derived  category $\bb{D}(\zz)$ with the $2m$-duality involution $C \mapsto C^{2m-*}.$ So there is a chain approximation on the graded chain complex
$$G_*C(X) \xrightarrow{G_*C(\Delta^X )} G_*(C(X) \otimes_{\bb{Z}} C(X)).$$





\begin{theorem}{\rm (Korzeniewski \cite[Theorem 4.5]{Korzen})}\qua \label{chain-iso} Given a fibration $F^{2m} \to E^{2m+2n} \to B^{2n}$ with $F$ and $B$ having the homotopy type of $CW$-complexes, let $X$ be a filtered space homotopy equivalent to $E$, with the filtration induced by the cellular structure of the base. Then there is
\begin{itemize}
\item[(i)] a chain isomorphism of chain complexes in the derived category $\bb{D}_{2m}(\bb{Z})$
$$\lambda \co G_*C(X)  \cong C(\widetilde{B}) \otimes_{\Z[\pi_1(B)]} (C(F), U).$$
\item[(ii)] a derived symmetric structure on $G_*C(X)$ obtained from the symmetric structure defined on $C(\widetilde{B}) \otimes_{\Z[\pi_1(B)]}  (C(F), U)$ and the chain isomorphism from part (i) of this Theorem.
\end{itemize}
\end{theorem}

\begin{proof}

\begin{itemize}
\item[(i)]
Denoting the $k$-cells in $B$ by an ordered set $I_k$ and the $k$-th skeleton of $B$ by $B_k$, we first observe that there is a commutative diagram, where the horizontal arrows are isomorphisms in $\mathbb{D}_{2m}(\zz)$
$$
\xymatrix{ \bigoplus_{j \in I_{k}} C(F) \ar[r] \ar[d] & G_kC(X) \ar[d]^{G_*d} \\
\bigoplus_{j \in I_{k-1}}C(F) \ar[r] & G_{k-1}C(X).
}$$
We can now construct the  following diagram
$$
\xymatrix{ \bigoplus_{j \in I_k} C(F)  \ar@/^2pc/[rr]^{Id} \ar[r] \ar[d] &  G_kC(X) \ar[d]^{G_*d} \ar[r]^{\hspace{-20pt} \lambda_k} & C_k(\widetilde{B}) \otimes (C(F), U) \ar[d]^{d_{C(\widetilde{B}) \otimes (C(F), U)}} \\
\bigoplus_{j \in I_{k-1}}C(F) \ar@/^-2pc/[rr]_{Id} \ar[r] & G_{k-1}C(X) \ar[r]^{\hspace{-20pt}\lambda_{k-1}} & C(\widetilde{B})_{k-1} \otimes (C(F), U). \\
}
$$
\raggedbottom

Here the square on the left
$$\xymatrix{\bigoplus_{j \in I_k} C(F) \ar[r] \ar[d] & G_kC(X) \ar[d]^{G_*d} \\
\bigoplus_{j \in I_{k-1}} C(F)  \ar[r] & G_{k-1}C(X)
} $$
commutes. The outer square also commutes
$$\xymatrix{\bigoplus_{j \in I_k} C(F) \ar[r] \ar[d] & C_k(\widetilde{B}) \otimes (C(F), U) \ar[d]^{d_{C(\widetilde{B})\otimes (C(F), U)}} \\
\bigoplus_{j \in I_{k-1}} C(F)  \ar[r] &C(\widetilde{B})_{k-1} \otimes (C(F), U).
} $$

Since the left and the outer squares commute, we deduce that the right square
$$
\xymatrix{  G_kC(X) \ar[d]_{G_*d} \ar[r]^-{\lambda_k} & C_k(\widetilde{B}) \otimes (C(F), U) \ar[d]^{d_{C(\widetilde{B}) \otimes (C(F), U)}} \\
G_{k-1}C(X) \ar[r]^-{\lambda_{k-1}} & C(\widetilde{B})_{k-1} \otimes (C(F), U) \\
}
$$
\raggedbottom
also commutes. Hence $\lambda \co G_{k}C(X) \to C(\widetilde{B})_{k} \otimes (C(F), U)$ is a chain map.

Using the same argument we can construct a chain map
$$\epsilon \co C(\widetilde{B}) \otimes (C(F), U) \to G_kC(X)$$
which is inverse to $\lambda$ if we view both $\epsilon$ and $\lambda$ as chain maps between chain complexes in $\mathbb{D}_{2m}(\zz)$.

\item[(ii)]
We have already proved that there is a chain equivalence $\lambda \co G_*C(X) \to C(\widetilde{B}) \otimes (C(F), U)$ and we also already described the chain approximation $ G_*C (\Delta^X) \co G_*C(X) \to G_*(C(X) \otimes_{\bb{Z}} C(X))$. The commutativity of the following diagram was proved in \cite[Theorem 3.11 and 4.5]{Korzen}, 
{\small
$$\xymatrix{
G_*C(X) \ar[r]^{\lambda} \ar[d]_{G_*C(\Delta^X)} & C(\widetilde{B}) \otimes (C(F), U) \ar[d]^{\Delta^{\widetilde{B}} \otimes \Delta^{F}}\\
G_*(C(X) \otimes_{\bb{Z}} C(X)) \ar[d]_{\theta^{X, X}}& (C(\widetilde{B}) \otimes_{\bb{Z}[\pi_1(B)]} C(\widetilde{B})) \otimes (C(F) \otimes_{\bb{Z}} C(F), U \otimes U) \ar[d]^{\theta^{\widetilde{B}, F}} \\
\textnormal{Hom}_{\bb{D}_{2m} (\zz)}\left((G_*C(X))^*,  G_*C(X) \right) \ar[r]^{\hspace{-30pt}\lambda \otimes \lambda} & (C(\widetilde{B}) \otimes (C(F), U)) \otimes_{\bb{Z}}(C(\widetilde{B}) \otimes (C(F), U)).
}
$$
}

The filtration-preserving symmetric structure on $G_*C(X)$, which we will denote later on in Section \ref{Filtered-signature} as $G_*\phi$, is obtained from the symmetric structure  $\phi^{C(\widetilde{B}) \otimes (C(F), U, \alpha)}$ on $C(\widetilde{B}) \otimes (C(F), U)$ described in Definition \ref{symm-str-induced-by-U} and the commutativity of the diagram above as 
$$\lambda \circ G_*\phi \circ \lambda^* = \phi^{C(\widetilde{B}) \otimes (C(F), U, \alpha)}.$$ The details for this can be found on \cite[page 42]{Korzen}.

\proved
\end{itemize}
\end{proof}

\subsection{The signature in the derived category} \label{Filtered-signature}

In Theorem \ref{chain-iso} we have seen that there is an equivalence  
of the associated graded complex of the filtered chain complex of the total space and the 
the tensor product $C(\widetilde{B}) \otimes (C(F), \alpha, U)$. The filtration preserving symmetric structure on $G_*C(X)$ was also described there in terms of the symmetric structure on the tensor product given in Definition \ref{symm-str-induced-by-U}.

In \cite[Lemma 4.10, Lemma 4.12]{Korzen} Korzeniewski describes a signature map for $2n$-dimensional symmetric chain equivalences $(G_*C, G_*\phi)$ in $\bb{D}_{2m}(\zz)$. Here $2m+2n \equiv 0 \pmod{4}.$ This signature map extends to a well-defined map
$$\sigma_{\bb{D}_{2m}(\zz)} \co L^{2n}(\bb{D}_{2m}(\zz)) \to \zz.$$
In this section we review the proofs by Korzeniewski relevant for this result, which we summarize in Proposition \ref{L-derived}.

\begin{proposition} \label{L-derived} {Korzeniewski \cite[Lemma 4.10, Lemma 4.12]{Korzen}} \qua
\begin{itemize}
\item[(i)] There is a well-defined homomorphism $\sigma_{\bb{D}_{2m}(\bb{R})}$ from $L^{2n}(\bb{D}_{2m}(\bb{R}))$ to $\zz.$
\item[(ii)]For any $2n$-filtered chain complex $C$ over the reals with a filtration-preserving nondegenerate symmetric structure $\phi$ of degree $2m+2n$, the signature of $(C, \phi)$ is equal to the derived signature of the associated graded chain complex $(G_*C, G_*\phi),$
$$\sigma(C, \phi) = \sigma_{\bb{D}_{2m}(\zz)}(G_*C, G_*\phi).$$
\end{itemize}
\end{proposition}
\begin{proof}
\begin{itemize}
\item[(i)]

An element in $L^{2n}(\mathbb{D}_{2m}(\mathbb{R}))$ is represented by $(C, \phi)$ where $C= (C_0 \leftarrow C_1 \leftarrow \dots \leftarrow C_{2n})$ is a chain complex in $\mathbb{D}_{2m}(\mathbb{R})$, so that each $C_r$ is an object of $\mathbb{D}_{2m}(\mathbb{R})$, and $\phi$ is a non-degenerate symmetric structure.
Without loss of generality, each $C_r$ has the form of a chain complex of real vector spaces with zero differential, that is $C_r \cong H_*C_r$.
Then from a symmetric chain complex $(C, \phi) \in L^{2n}(\mathbb{D}_{2m}(\mathbb{R}))$ we can form a new symmetric chain complex $(C', \phi')$ with each $C'_r$ defined as
$$C'_r = H_rC = \frac{\textnormal{Ker}~(d \co C_r \to C_{r-1})}{\textnormal{Im}~(d \co C_{r+1} \to C_r)} $$

This has trivial differentials,
$$\dots \xrightarrow{0} C'_r \xrightarrow{0} C'_{r-1} \xrightarrow{0} \dots  $$

The map in the middle dimension 
$$\phi' \co  (C'_n)^* \to C'_n $$
is a symmetric form over $\mathbb{R}.$

The chain complex $C$ in $\mathbb{D}_{2m}(\mathbb{R})$ is chain equivalent to $C'$ in $\mathbb{D}_{2m}(\mathbb{R})$, so $(C, \phi)$ is equivalent to $(C', \phi')$.
The signature $\sigma_{\mathbb{D}_{2m}(\mathbb{R})}(C, \phi)$ is defined to be the signature of the symmetric form $\phi' \co (C'_n)^* \to C'_n.$

The map $\sigma_{\mathbb{D}_{2m}(\mathbb{R})}$ is clearly additive, that is, if $(C,\phi^C)$ and $(B, \phi^B)$ are two chain complexes in $L^{2n}(\mathbb{D}_{2m}(\mathbb{R}))$, then
$$\sigma_{\bb{D}_{2m}(\bb{R})}\left(C \oplus B,\phi^C \oplus \phi^B \right) = \sigma_{\bb{D}_{2m}(\bb{R})}(C, \phi^C) + \sigma_{\bb{D}_{2m}(\bb{R})}(B, \phi^B).$$
 The proof that $\sigma_{\bb{D}_{2m}(\bb{R})}(C, \phi) = 0$ for elements representing zero in $L^{2n}(\mathbb{D}_{2m}(\mathbb{R}))$ is given on \cite[page 52]{Korzen}.

\item[(ii)]
This result is proved by Korzeniewski in \cite[page 54]{Korzen}.  \proved
\end{itemize}
\end{proof}

Using part (i) in  Proposition \ref{L-derived}, we see that  for a fibration $F^{2m} \to E \xrightarrow{p} B^{2n}$ the composite
$$p^!\co L^{2n}(\bb{Z}[\pi_1(B)]) \xrightarrow{-\otimes (C(F), \alpha, U)} L^{2n}(\bb{D}_{2m}(\bb{Z}))  \xrightarrow{\sigma_{\bb{D}_{2m} (\bb{Z})}} L^{2m+2n}(\zz)= \bb{Z}$$
is a transfer map in symmetric $L$-theory.

If $B$ is a $2n$-dimensional geometric Poincar\'e complex then $E$ is a $(2m+2n)$-dimensional geometric Poincar\'e complex with the transfer of the
symmetric signature $\sigma^*(B) \in L^{2n}(Z[\pi_1(B)])$ the ordinary signature
$\sigma(E) = p^!(\sigma^*(B)) \in L^{2m+2n}(\zz) = \zz.$

\begin{example} With this example we shall illustrate Proposition  \ref{L-derived} in two special cases, when the base is a point and when the fibre is a point.
\begin{itemize}
\item Case 1: Let $F^{2m} \to E \to \{pt\} $ be a fibration with base a point, that is, we take $n=0$. Then $F \to E$ is a homotopy equivalence, and there is a symmetric $L$-theory transfer
$$p^! \co L^0(\zz) \to L^0(\bb{D}_{2m}(\zz)) \to L^{2m}(\zz)=\zz.$$

In general the ring morphism $\zz[\pi_1(B)] \to H_0( \textnormal{Hom}_{\zz}(C(F), C(F))^{op})$ induces a map in $L$-theory $L^0(\zz[\pi_1(B)]) \to L^0 (H_0( \textnormal{Hom}_{\zz}(C(F), C(F))^{op}))$. Composing this with the canonical map  $L^0(H_0(\textnormal{Hom}_{\zz}(C( F ),C( F ))^{op})) \to L^0(\bb{D}_{2m}(\zz))$ we obtain the first functor in the transfer map.
If $B= \{ pt \}$, the ring morphism is $$\zz \to H_0( \textnormal{Hom}_{\zz}(C( F ),C( F ))^{op}); 1 \mapsto 1,$$ and the canonical map sends $1$ to $C(F)$ in the $0$-th filtration.
Therefore the transfer map is
\begin{gather*}
p^!\co L^0(\zz)=\zz \to L^{2m}(\zz) ; 1 \mapsto \sigma(E)=\sigma(F) \\
\tag*{\text{and}} p^!(\sigma^*(B)) = \sigma(F)= \sigma(E).
\end{gather*}

\item Case 2:  Let $\{pt\} \to E \to B^{2n}$ be a fibration with fibre a point, that is we take $m=0$. Then $p \co E \to B$ is a homotopy equivalence and
$$p^! \co L^{2n}(\zz[\pi_1(B)]) \to L^{2n}(\zz)=\zz$$
is the forgetful map induced by the augmentation $\zz[\pi_1(B)]\to \zz,$ and
$$p^!(\sigma^*( B)) = \sigma(B)= \sigma( E).$$
\end{itemize}
\end{example}

At this point we know that for a fibration $F^{2m} \to E \to B^{2n}$ with $2m+2n \equiv 0 \pmod{4}$
 $$\sigma (E) = \sigma(X) =  \sigma_{\bb{D}_{2m}({\zz})}(G_*C(X)) \in \zz,$$
where $X$ is a filtered complex homotopy equivalent to $E$. And we have also described a well defined functor $$L^{2n}(\bb{Z}[\pi_1(B)])\xrightarrow{-\otimes (C(F), \alpha, U)} L^{2n}(\bb{D}_{2m} (\bb{Z}))  \xrightarrow{\sigma_{\bb{D}_{2m} {(\bb{Z})}}} \bb{Z}.$$
To show that this functor describes the signature of the total space just need to observe that $\sigma_{\bb{D}_{2m}({\zz})}(G_*C(X)) = \sigma_{\bb{D}_{2m}({\zz})}(C(\widetilde{B}) \otimes (C(F), \alpha, U)). $ 
%
%
%
%
Using the results of Theorem \ref{E and X},  Proposition \ref{L-derived} and    Theorem \ref{chain-iso} we see that
$$\sigma (E) = \sigma(X) =  \sigma_{\bb{D}_{2m}({\zz})}(G_*C(X)) = \sigma_{\bb{D}_{2m}({\zz})}((C(\widetilde{B}), \phi) \otimes (C(F), \alpha, U)) \in \bb{Z}.$$

\subsection{Two equivalent functors for the signature of a fibration}\label{functors-sec}

The definition of a symmetric representation was given in Definition \ref{sym-rep}.

Let $(C(F), \alpha, U)$ be the symmetric representation of the group ring $\bb{Z}[\pi_1(B)]$ in $\bb{D}_{2m}(\bb{Z})$ associated to the fibration $F^{2m} \to E \to B^{2n}$.
\begin{definition}\label{functor-chain}
The symmetric representation $(C(F), \alpha, U)$ gives rise to the following functor of additive categories with involution,
$$- \otimes (C(F), \alpha, U) \co  \bb{B}(\bb{Z}[\pi_1(B)]) \to \bb{D}_{2m}(\bb{Z}).$$
\end{definition}
 From the chain symmetric representation  $(C(F), \alpha, U)$ we can construct  a homology symmetric representation $(A, \alpha, U)$ associated to the same fibration which is given by
$$
\begin{array}{l}
A = H_m(C(F))/ torsion, \\
\alpha \co  A=   H_m(C(F))/ torsion \to A^*= H^m(C(F))/ torsion, \\
U  \co \bb{Z}[\pi_1(B)] \to H_0(\textnormal{Hom}_{\bb{Z}}(A, A))^{op}.
\end{array}
$$

\begin{rem} \label{skew-susp} Note that $(A, \alpha) \in L^0(\bb{Z}, (-1)^{m})$, so the tensor product $(C(\widetilde{B}), \phi) \otimes_{\zz} (A, \alpha, U)$ does not immediately give us  a chain complex of the dimension of the total space. To obtain the correct dimension we just need to use skew suspension of $(A, \alpha),$
$$ \bar{S}^{m} \co L^0(\zz, (-1)^m) \to L^{2m}(\zz, (-1)^{2m}) ; ~(A, \alpha) \mapsto  \bar{S}^{m}(A, \alpha) $$

\end{rem}



\begin{definition} \label{representation A}
The homology symmetric representation $ \bar{S}^{m}(A, \alpha, U)$ associated to the fibration $F^{2m} \to E \to B^{2n}$ is the $(-1)^m$-symmetric form $(A, \alpha)$ together with a representation $U$ of the group ring $\bb{Z}[\pi_1(B)]$ in the additive category with involution $\bbD_{2m}(\zz)$.
\end{definition}

\begin{definition}\label{functor-homology}
The homology symmetric representation $ \bar{S}^{m}(A, \alpha, U)$ gives rise to the following functor of additive categories with involution
$$- \otimes  \bar{S}^{m} (A, \alpha, U) \co \bb{B}(\bb{Z}[\pi_1(B)]) \to \bbD_{2m}(\zz).$$

\end{definition}

The two functors from Definitions \ref{functor-chain} and \ref{functor-homology} induce maps in symmetric $L$-theory,
\begin{gather*} - \otimes (C(F), \alpha, U) \co L^{2n}(\bb{Z}[\pi_1(B)]) \to L^{2n}(\bb{D}_{2m}(\bb{Z})) \\
\tag*{\text{and}} - \otimes \bar{S}^{m}(A, \alpha, U) \co L^{2n}(\bb{Z}[\pi_1(B)]) \to L^{2n}(\bb{D}_{2m}(\bb{Z})).
\end{gather*}

\begin{proposition} \label{two-functors}
Using the functors for a fibration $F^{2m} \to E \to B^{2n}$ described in Definitions \ref{functor-chain} and \ref{functor-homology}, the following diagram commutes
\begin{displaymath}
\xymatrix{
 L^{2n}(\bb{Z}[\pi_1(B)]) \ar[d]_{=}\ar[rr]^{- \otimes (C(F), \alpha, U)} & &  L^{2n}(\bb{D}_{2m}(\bb{Z}))\ar@/^2.0pc/[rrr]^{\sigma_{\bb{D}_{2m}(\bb{Z})}} \ar[r]^{- \otimes \bb{R}} & L^{2n}(\bb{D}_{2m}(\bb{R})) \ar[rr]^{\sigma_{\bb{D}_{2m}(\bb{R})}} &&  \zz \ar[d] \\
 L^{2n}(\bb{Z}[\pi_1(B)]) \ar[rr]^{- \otimes  \bar{S}^{m}(A, \alpha, U)} & &  L^{2n}(\bb{D}_{2m}(\bb{Z}))\ar@/^-2.0pc/[rrr]^{\sigma_{\bb{D}_{2m}(\bb{Z})}} \ar[r]^{- \otimes \bb{R}} & L^{2n}(\bb{D}_{2m}(\bb{R})) \ar[rr]^{\sigma_{\bb{D}_{2m}(\bb{R})}} &&  \zz.
}
\end{displaymath}

\end{proposition}

\begin{proof}
$$
\begin{array}{ccl}
\sigma_{\bb{D}_{2m}(\bb{Z})}(-\otimes (C, \alpha, U)) & = & \sigma_{\bb{D}_{2m}(\bb{R})}(-\otimes \bar{S}^{m} (H_*(C\otimes \bb{R}), \alpha \otimes \bb{R}, U\otimes \bb{R})) \\
 & = & \sigma_{\bb{D}_{2m}(\bb{R})}(-\otimes  \bar{S}^{m} (H_0(C\otimes \bb{R}) \oplus (H_{2m}(C\otimes \bb{R}), \alpha \otimes \bb{R}, U\otimes \bb{R})) + \\
& &  \sigma_{\bb{D}_{2m}(\bb{R})}(-\otimes  \bar{S}^{m}(H_1(C\otimes \bb{R}) \oplus (H_{2m-1}(C\otimes \bb{R}), \alpha \otimes \bb{R}, U\otimes \bb{R})) + \\
& &  \dots + \sigma_{\bb{D}_{2m}(\bb{R})}(-\otimes S^{2m}(H_m(C\otimes \bb{R}), \alpha \otimes \bb{R}, U\otimes \bb{R})).
\end{array}
$$
The signature only depends on the middle homology so the only non-zero term is $\sigma_{\bb{D}_{2m}(\bb{R})}(-\otimes  \bar{S}^{m}(H_m(C\otimes \bb{R}), \alpha \otimes \bb{R}, U\otimes \bb{R}))$ the other terms are just hyperbolic modules which are $0 \in L^{2m+2n}(\bb{Z})$. hence,
$$
\begin{array}{ccl}
\sigma_{\bb{D}_{2m}(\bb{Z})}\left(-\otimes (C, \alpha, U)\right) & = & \sigma_{\bb{D}_{2m}(\bb{R})} \left(-\otimes  \bar{S}^{m} (H_m(C\otimes \bb{R}), \alpha \otimes \bb{R}, U\otimes \bb{R})\right) \\
& = & \sigma_{\bb{D}_{2m}(\bb{R})} \left(- \otimes  \bar{S}^{m}(A \otimes \bb{R}, \alpha \otimes \bb{R}, U \otimes \bb{R})\right) \\
& = & \sigma_{\bb{D}_{2m}(\bb{Z})} \left(- \otimes  \bar{S}^{m}(A, \alpha , U)\right).
\end{array}
$$
and the result follows.\proved
\end{proof}

\begin{rem} \label{equal-signatures}
We had already noted that
$$\sigma(E) = \sigma_{\bb{D}_{2m}(\bb{Z})} (C(\widetilde{B}), \phi) \otimes (C(F), \alpha, U). $$
Combining this result with Proposition \ref{two-functors} we have that
$$\sigma(E) = \sigma((C(\widetilde{B}), \phi) \otimes  \bar{S}^{m} (A, \alpha, U)).$$
\end{rem}

\section{The signature and the Arf and Brown-Kervaire invariants} \label{Sign-Arf-BK}
\subsection{Pontryagin squares} \label{P-sq}

\subsubsection{Cup-$i$ products}\label{cup-i}
To define Pontryagin squares it is first necessary to introduce cup-$i$ products. The construction of the cup-$i$ products is defined in detail by Mosher and Tangora in \cite{Mosher-Tangora}.
\begin{definition} {\rm (Mosher and Tangora \cite{Mosher-Tangora})}\qua \label{cup-i products} For each integer $i \geq 0$, define a cup-$i$ product
\begin{gather*}\cup_i \co C^p(X) \otimes C^q(X) \longrightarrow C^{p+q-i}(X) \co (u, v) \longrightarrow u \cup_i v \\
\tag*{\text{by the formula}}
(u \cup_i v)(c)  = (u \otimes v) \phi(d_i \otimes c),
\end{gather*}
where $c \in C_{p+q-i}(X)$ and $d_i$ is the standard generator of $W_i = \mathbb{Z}[\mathbb{Z}_2]$ and
\begin{gather*}\phi \co W \otimes C(X) \to C(X) \otimes C(X)
\end{gather*} is an equivariant  Eilenberg-Zilber diagonal map.
\end{definition}

More details about the equivariant chain map $\phi$ and the chain complex $W$ are given in Mosher and Tangora \cite[Chapter 2]{Mosher-Tangora}.


\subsubsection{Classical Pontryagin squares} \label{class-pont}

Pontryagin first defined the cohomology operation known as Pontryagin square in \cite{Pontryagin}. Pontryagin squares were also carefully studied by Whitehead in \cite{Whitehead2} and \cite{Whitehead}.
With $X$ a space, the Pontryagin square is an unstable cohomology operation
$$\mc{P}_2 \co H^n(X; \bb{Z}_2) \to H^{2n}(X; \bb{Z}_4).$$
Although the Pontryagin square is defined on modulo $2$ cohomology classes, it cannot be constructed solely from the modulo $2$ cup product structure.

Let  $d^* \co C^n(X; \bb{Z}) \to C^{n+1}(X; \bb{Z})$  be the singular cohomology coboundary operator.
We shall represent an element $x \in H^n(X ; \zz_2)$ as a cycle
$$x=(y, z) \in \textnormal{Ker} \left(\left(\begin{array}{cc} d^* & 2 \\   0 & d^* \end{array}    \right) \co C^n(X) \oplus C^{n+1}(X) \to C^{n+1}(X) \oplus C^{n+2}(X)\right),$$
that is, $d^*(z)= 0$ and $d^*(y) + 2z =0.$
Using this notation we define the Pontryagin square as follows.

\begin{definition}
  The Pontryagin square is defined on the cochain level by
\begin{align*}
\mc{P}_2(x) = \cP_2(y, z) &= y \cup_0 y + y \cup_1 d^*y  \\
                                       & = y \cup_0 y + 2 y \cup_1 z,
\end{align*}
where $y \cup_1 z$ is the cup-$1$ product.
\end{definition}

The coboundary formula for the cup-$i$ product, with $u \in C^p(X)$ and $v \in C^q(X)$ is given by
$d(u \cup_i v) = (-1)^i du \cup_i v + (-1)^{i+p}u \cup_i dv + (-1)^{i+1} u \cup_{i-1} v + (-1)^{pq+1} v \cup_{i-1} u.$
This formula can be applied to check that $y \cup_0 y + y \cup_1 d^*y \textnormal{ mod } 4 $ is a cocycle mod $4$ and that its cohomology class only depends on that of $x = (y, z) \in H^n(X; \bb{Z}_2)$ (see Mosher and Tangora \cite{Mosher-Tangora}).

\begin{definition}\label{Pont-square-definition}
 The Pontryagin square is defined on cohomology by
\begin{align*}\mc{P}_2 \co H^n(X ; \bb{Z}_2) & \to H^{2n}(X; \bb{Z}_4) \\
                                                 x = (y, z) & \mapsto y \cup_0 y + y \cup_1 d^*y,
\end{align*}
where $(y, z)$ are as defined above.
\end{definition}

The maps in the exact sequence $0 \longrightarrow \mathbb{Z}_2 \overset{i}{\longrightarrow} \mathbb{Z}_4 \overset{r}{\longrightarrow} \mathbb{Z}_2 \longrightarrow  0 $ induce maps in cohomology,
\begin{gather*}\dots \longrightarrow H^{2n}(X; \mathbb{Z}_2)\overset{i_*}{ \longrightarrow} H^{2n}(X; \mathbb{Z}_4) \overset{r_*}{\longrightarrow} H^{2n}(X; \mathbb{Z}_2) \overset{\delta}{\longrightarrow} \dots \\
\tag*{\text{so}}
r_* \mathcal{P}_2(x) = x \cup x \in H^{2n} (X ; \mathbb{Z}_2),
\end{gather*}
 where $r_* \co H^{n}(X ; \mathbb{Z}_4) \longrightarrow H^{n}(X ; \mathbb{Z}_2)$ is the map induced by the non-trivial map \\ $\mathbb{Z}_4 \longrightarrow \mathbb{Z}_2$.

\begin{rem}
The Steenrod square $Sq^{n}$ is a mod $2$ reduction of the Pontryagin square:
\begin{displaymath}
\xymatrix{ &  H^{2n}(X; \mathbb{Z}_4) \ar[d]^{r_*}  \\
 H^{n}(X; \mathbb{Z}_2) \ar[r]^{Sq^{n}} \ar[ur]^{\mathcal{P}_2} &  H^{2n}(X; \mathbb{Z}_2).
}
\end{displaymath}
\end{rem}

\begin{proposition}{(Mosher and Tangora \cite{Mosher-Tangora})} \qua \label{on a sum}
\begin{itemize}
\item[(i)]Let $x$ and $x'$ be cocycles in $H^n(X; \bb{Z}_2)$, where $X$ is a  Poincar\'e space.  The Pontryagin square evaluated on a sum is given by
$$\mc{P}_2(x+x') = \mc{P}_2(x) + \mc{P}_2(x')+(x \cup x')+(-1)^n(x \cup x') \in H^{2n}(X; \bb{Z}_4).$$
So that in particular for $n$ even
$$\mc{P}_2(x+x') = \mc{P}_2(x) + \mc{P}_2(x')+i (x \cup x')\in \bb{Z}_4,$$
where $i$  is the non-trivial homomorphism $i \co \bb{Z}_2 \to \bb{Z}_4.$

\item[(ii)] The Pontryagin square is a quadratic function with respect to cup product, that is,
$$\mc{P}_2(x+ x') = \mc{P}_2(x) + \mc{P}_2(x') + i(x \cup x') \in H^{2n}(X; \bb{Z}_4),$$
for any $x, x' \in H^{n}(X ; \bb{Z}_2)$, where $i \co \bb{Z}_2 \to \bb{Z}_4$ is the non-trivial homomorphism.
\end{itemize}
\end{proposition}

\subsubsection{Algebraic Pontryagin squares}\label{alg-pont}
An algebraic analog of the Pontryagin square can be constructed as follows. The algebraic Pontryagin square lies in the $4k$-dimensional symmetric $Q$-group of a finitely generated free $\bb{Z}$-module chain complex concentrated in degrees $2k+1$ and $2k$
$$B(2k, 2) \co \dots \to 0  \to S^{2k+1}\bb{Z} \xrightarrow{d= 2} S^{2k}\bb{Z} \to 0 \to \dots.$$
In  \cite{BanRan} Banagl and Ranicki give formulas for the computations of  $Q$-groups  and in \cite{Mod8} Ranicki and Taylor show explicitly that $Q^{4k}(B(2k, 2))$ is given by the isomorphism $$Q^{4k}(B(2k, 2)) \to \bb{Z}_4 ; ~ \phi \to \phi_0 + d\phi_1.$$

\begin{definition}{(Ranicki and Taylor \cite{Mod8})} \qua \label{algebraic-Pont}
The \textbf{$\zz_2$-coefficient Pontryagin square} of a $4k$-dimensional symmetric complex $(C, \phi)$ over $\zz$ is the function
$$
\begin{array}{rcl}
\mc{P}_2(\phi) \co H^{2k}(C ; \zz_2) = H_0(\textnormal{Hom}_{\zz}(C, B(2k,2)))& \to &Q^{4k}(B(2k, 2)) = \zz_4\\
(u, v) & \mapsto & (u, v)^{\%}(\phi) = \phi_0(u, v) + 2\phi_1(v, u).
\end{array}
$$
\end{definition}

In definition \ref{algebraic-Pont},  $C$ is a $4k$-dimensional chain complex with a $4k$-dimensional symmetric structure, $B(2k, 2)$ is a finitely generated free $\bb{Z}$-module chain complex concentrated in degrees $2k+1$ and $2k$ and $g$ is a chain map $g = (u, v) \co C \longrightarrow B(2k, d)$.

\begin{rem}
If $C=C(X)$ is the chain complex of a space $X$ and  $\phi=\phi_X[X] \in Q^{4k}(C)$ is the image of a homology class $[X] \in H_{4k}(X)$ under the symmetric construction $\phi_X$,
 then the evaluation of the Pontryagin square $\mathcal{P}_2 \co H^{2k}(X;\zz_2) \to H^{4k}(X;\zz_4)$ on the mod $4$ reduction $[X]_4 \in H_{4k}(X;\zz_4)$ is the algebraic Pontryagin square
$$\mathcal{P}_2(\phi) \co H^{2k}(C; \zz_2)=H^{2k}(X; \zz_2) \to \zz_4 ; x \mapsto \langle \mathcal{P}_2(x), [X]_4 \rangle.$$
There is a commutative diagram
$$
\xymatrix{ H^{2k}(C(X); \zz_2) \ar[r]^{\mc{P}_2(\phi)} &  Q^{4k}(B(2k,2))= \zz_4 \\
H^{2k}(X ; \zz_2) \ar[u]^{=} \ar[r]^{\hspace{-55pt} \mc{P}_2} & H^{4k}(X; \zz_4)= H^{4k}(X ; Q^{4k}(B(2k,2))) \ar[u]_{\langle [X]_4, - \rangle}.
}
$$
We are mainly interested in the case when $X$ is a $4k$-dimensional geometric Poincar\'e space and $[X]$ is the fundamental class.
\end{rem}

\subsection{The signature modulo $4$ and modulo $8$} \label{Morita-theorem}

\subsubsection{The signature modulo $4$} \label{sign-mod4}
A result from Morita \cite{Morita} that will be relevant for us gives the relation between the signature modulo $4$ and the Pontryagin square. The formulation by Morita in \cite{Morita} is for an oriented Poincar\'e space, and we will use an algebraic analog.

\begin{proposition}{(Morita \cite[Proposition 2.3]{Morita})} \qua \label{mod4wu}
Let $X^{4k}$ be an oriented Poincar\'e space and $\mc{P}_2 \co H^{2k}(X ; \bb{Z}_2) \to H^{4k}(X ; \bb{Z}_4)$ be the Pontryagin square, then
$$\sigma(X) = \langle \mc{P}_2 (v_{2k}), [X] \rangle \in \zz_4, $$
where $v_{2k} \in H^{2k}(X ; \bb{Z}_2)$ is the $2k$-th Wu class of $X.$
\end{proposition}
The algebraic analogue of Proposition \ref{mod4wu} is stated as follows:
\begin{proposition} \label{signature-Psq}
Let $(C, \phi)$ be a $4k$-dimensional symmetric Poincar\'e complex over $\mathbb{Z}$, then
$$\sigma(C, \phi) = \mc{P}_2(v_{2k})\in \zz_4,$$
where  $v_{2k} \in H^{2k}(C ; \bb{Z}_2)$ is the $2k$-th algebraic Wu class of  $(C, \phi).$
\end{proposition}

\subsubsection{The signature modulo $8$} \label{Sign-mod8}

In \cite[theorem~1.1]{Morita} Morita relates the Brown-Kervaire invariant and the signature modulo $8$.

Let $V$ be a $\bb{Z}_2$ vector space, $\lambda \co V \otimes V \to \bb{Z}_2$ a non-singular symmetric pairing and let $q \co V \to \bb{Z}_4$ be a quadratic enhancement of the symmetric form so that
$$q(x +y) = q(x) + q(y)+ i \lambda (x, y) \in \bb{Z}_4,$$
where $i = 2 \co \bb{Z}_2  \to \bb{Z}_4$ and $x, y \in V$.
\begin{definition}{(Brown \cite{Brown})}\qua \label{Gauss-sum-formula}
The \textbf{Brown-Kervaire $\mathrm{BK}(V, \lambda, q)$ invariant} is defined using a Gauss sum
$$\sum_{x \in V} i ^{q(x)} = \sqrt{2}^{\textnormal{dim}V} e^{2 \pi i \textnormal{BK}(V,  \lambda, q) / 8},$$
with $i^2=-1$ and $x \in V.$
\end{definition}

Two non-singular $\zz_4$-valued quadratic forms on a $\zz_2$-vector space V of finite dimension are Witt equivalent if and only if they have the same Brown-Kervaire invariant.

The theorem by Morita is formulated geometrically and it relates the signature of a $4k$-dimensional Poincar\'e space and the Brown-Kervaire invariant of the Pontryagin square, which is a quadratic enhancement of the cup product structure on the $\bb{Z}_2$-vector space $H^{2k}(X; \mathbb{Z}_2)$.

Morita's theorem is given in \cite{Morita} as follows:

\begin{theorem}{\rm (Morita \cite[theorem~1.1]{Morita})} \qua \label{Morita}
Let $X$ be a $4k$-dimensional Poincar\'e space, then
$$\sigma (X)  = \mathrm{BK}(H^{2k}(X; \mathbb{Z}_2), \lambda, \mc{P}_2) \in \zz_8.$$
\end{theorem}

In the proofs of the main results of this paper we will require a reformulation of Morita's theorem in terms of symmetric Poincar\'e complexes $(C, \phi)$. The Pontryagin square which geometrically depends on the cup and cup-$1$ products, depends algebraically on the symmetric structure $\phi$ as was explained in the previous section. This is denoted by $\mc{P}_2(\phi)$, although for simplicity we will  write $\mc{P}_2$ when it is clear from the context.
\begin{theorem} \label{Morita-theorem-cx}
Let $(C, \phi)$ be a $4k$-dimensional symmetric Poincar\'e complex over $\mathbb{Z}$, then
$$\sigma (C, \phi) = \mathrm{BK}(H^{2k}(C; \mathbb{Z}_2), \phi, \mathcal{P}_2)  \in \zz_8.$$
\end{theorem}

When the signature modulo $8$ is divisible by $4$, it can be expressed as an Arf invariant of a certain quadratic form.

\begin{theorem}\label{4Arf-topology}An $4k$-dimensional symmetric Poincar\'e complex $(C, \phi)$ has signature $0$ mod $4$ if and only if
$L=\langle v_{2k}(C) \rangle \subset H^{2k}(C;\zz_2)$ is a sublagrangian of $(H^{2k}(C;\zz_2),\lambda,\mc{P}_2)$. If such is the case, there is
defined a sublagrangian quotient nonsingular symmetric form over $\zz_2$ with a $\zz_2$-valued enhancement
$$(W,\mu,h) = (L^{\perp}/L , [\lambda] , h = [\mc{P}_2]/2)$$
and the signature mod $8$ is given by
$$\sigma(C, \phi) = 4\textnormal{Arf}(W,\mu,h) \in  4\zz_2 \subset \zz_8.$$
\end{theorem}

\begin{proof}
By Morita \cite[Theorem 1.1]{Morita} we know that $$\sigma(C, \phi) = \mathrm{BK}(H^{2k}(C;\zz_2),\lambda,\mc{P}_2)  \subset \zz_8.$$
We need to prove that if $\mathrm{BK}(H^{2k}(C;\zz_2),\lambda,\mc{P}_2) \equiv 0 \pmod{4}$ then,
 $$\mathrm{BK}(H^{2k}(C;\zz_2),\lambda,\mc{P}_2)
= 4\textnormal{Arf}(W,\mu,h) \in  4\zz_2 \subset \zz_8.$$
$\left(H^{2k}(C;\zz_2),\lambda,\mc{P}_2\right) $ is a nonsingular symmetric form over $\zz_2$ with a $\zz_4$-valued quadratic enhancement given by the Pontryagin square. Denote its dimension by  $n = \textnormal{dim}(H^{2k}(C;\zz_2))$.
The Wu class $v=v_{2k}(C) \in H^{2k}(C;\zz_2)$ is such that
$$\lambda(x,x) = \lambda(x,v) \in \zz_2 \textnormal{ with } x\in H^{2k}(C;\zz_2).$$
The following function is linear
$$f \co H^{2k}(C;\zz_2) \longrightarrow \zz_2 ; x \mapsto \lambda(x,x) = j \mc{P}_2(x) = \lambda(x,v),$$
where $j\mc{P}_2(x)\in \zz_2$ is the mod $2$ reduction of $\mc{P}_2(x) \in \zz_4.$

The following identity \eqref{identity} relating the Pontryagin square and mod $4$ reduction of the Brown-Kervaire invariant is a consequence of Morita \cite[theorem 1.1]{Morita} and \cite[Proposition 2.3]{Morita}
\begin{equation}\label{identity}[\mathrm{BK}(H^{2k}(C;\zz_2),\lambda,\mc{P}_2)] = \mc{P}_2(v) \in \zz_4.\end{equation}
So $\mathrm{BK}(H^{2k}(C;\zz_2),\lambda,\mc{P}_2) \in \zz_8$ is divisible by $4$ if and only if $\mc{P}_2(v)=0 \in \zz_4$. \\
If $\mc{P}_2(v)=0 \in \zz_4$ then $\lambda(v,v) = 0 \in \zz_2$ and the \textit{Wu sublagrangian} $L=\langle v \rangle \subset (H^{2k}(C;\zz_2),\lambda,\mc{P}_2)$
is defined, with $L \subseteq L^\perp = \left\{x \in H^{2k}(C;\zz_2)|\lambda(x,x)=0 \in \zz_2\right\}$.
The \textit{maximal isotropic subquotient} $(L^{\perp}/L, [\lambda])$
 has a canonical $\zz_2$-valued quadratic enhancement $[h]\co x \mapsto [\mc{P}_2(x)]/2$ and
$$\mathrm{BK}(H^{2k}(C;\zz_2),\lambda,\mc{P}_2) = 4\mathrm{Arf}(L^{\perp}/L,[\lambda],[h]) \subset 4\zz_2 \subset \zz_8.$$
For the dimension of $L^{\perp}/L$ there are two cases, according as to whether $v=0$ or $v\neq 0:$
\begin{itemize}
\item[(i)] If the Wu class is $v=0$ then $(H^{2k}(C;\zz_2), \lambda)$ is already isotropic then $L^{\perp}/L = H^{2k}(C;\zz_2),$ and $\mathrm{dim}(L^{\perp}/L) = n.$
\item[(ii)] If the Wu class is $v \neq 0$ then $(H^{2k}(C;\zz_2), \lambda)$ is anisotropic and $\mathrm{dim}(L^{\perp}/L) = n-2.$
\proved
\end{itemize}
\end{proof}


\section{The signature of a fibration modulo $8$} \label{mod8-theorems}

\subsection{Obstructions to multiplicativity modulo 8 of a fibration} \label{obstructions}

By the results of Meyer in \cite{Meyerpaper} and of Hambleton, Korzeniewski and Ranicki in \cite{modfour} the signature of a fibration
$F^{2m} \to E^{4k} \to B^{2n}$ of  geometric Poincar\'e complexes is multiplicative mod $4$
$$\sigma(E) - \sigma(B) \sigma(F) =0 \in \zz_4 .$$
If we set $M = E \sqcup - (B \times F)$, where $-$ reverses the orientation, then $M$ has signature
$$\sigma(M) = \sigma(E) - \sigma(B)\sigma(F) \in \zz,$$
so that $\sigma(M) = 0 \in \zz_4$ and Theorem \ref{4Arf-topology} can be applied to $M$.

From Theorem \ref{4Arf-topology} we know that when the signature is divisible by $4$, it is detected modulo $8$ by the Arf invariant. This can be applied in the situation of the signature of a fibration.
\begin{theorem}\label{4Arf-general-fibration}
Let $F^{2m} \to E^{4k} \to B^{2n}$ be a   Poincar\'e duality fibration. With $(V, \lambda)= \left(H^{2k}(E, \zz_2), \lambda \right)$ and $(V', \lambda')= \left(H^{2k}(B \times F), \zz_2), \lambda' \right)$, the signatures mod $8$ of the fibre, base and total space are related by
$$
\sigma(E) - \sigma(B \times F) = 4 \mathrm{Arf}\left( L^{\perp}/L , [\lambda \oplus - \lambda'], \frac{\left[ \mc{P}_2 \oplus-\mc{P}'_2 \right]}{2} \right) \in 4\zz_2 \subset \zz_8,
$$
where $L^{\perp} =\left\{(x, x') \in V \oplus V' \vert \lambda(x,x) = \lambda'(x', x') \in \zz_2 \right\}$ and $L = \langle v_{2k} \rangle \subset L^{\perp}$, with $v_{2k}=(v_{2k}(E), v_{2k}(B \times F)) \in V \oplus V'$ the Wu class of $E \sqcup -(B \times F)$ and $\mc{P}_2$ and $\mc{P}'_2$ the Pontryagin squares of $E$ and $B \times F$ respectively.
\end{theorem}

\begin{proof}
We first rewrite the signatures $\sigma(E)$ and $\sigma(F \times B)$ in terms of Brown-Kervaire invariants,  and use the additivity properties of the Brown-Kervaire invariant described by Morita in \cite[Proposition 2.1 (i)]{Morita} as follows
\begin{align*}
\sigma(E) - \sigma(B \times F) &=\mathrm{BK}(H^{2k}(E;\zz_2), \lambda, \mc{P}_2) - \mathrm{BK}(H^{2k}(B \times F;\zz_2), \lambda', \mc{P}'_2) \in \zz_8 \\
& = \mathrm{BK}\left(V \oplus V', \lambda \oplus -\lambda', \mc{P}_2 \oplus - \mc{P}'_2 \right) \in \zz_8.
\end{align*}

We know by Meyer \cite{Meyerpaper} and by Hambleton, Korzeniewski and Ranicki \cite[Theorem A]{modfour} that
$$ \sigma(E) - \sigma(B \times F) = 0 \in \zz_4.$$
 Applying Theorem \ref{4Arf-topology}, $\mathrm{BK}\left(V \oplus V', \lambda \oplus -\lambda', \mc{P}_2 \oplus - \mc{P}'_2 \right) \in 4\zz_2 \subset \zz_8$ can be written as an Arf invariant,
$$4 \mathrm{Arf}\left( L^{\perp}/L , [\lambda \oplus - \lambda'],  \frac{\left[\mc{P}_2 \oplus-\mc{P}'_2\right] }{2} \right) \in \zz_8,$$
with $L^{\perp} =\left\{(x,x') \in V\oplus V' \vert \lambda (x, x)= \lambda'(x',x')  \in \zz_2 \right\}$ and the Wu sublagrangian $L = \langle (v_{2k}(E), v_{2k}(B \times F)) \rangle \subset L^{\perp}$, with $(v_{2k}(E), v_{2k}(B \times F))$ the  Wu class given by $(v_{2k}(E), v_{2k}(B \times F)) \in H^{2k}(E;\zz_2) \oplus H^{2k}(B \times F;\zz_2) = V \oplus V'. $\proved
\end{proof}

In the following theorem we state the chain complex version of Theorem \ref{4Arf-general-fibration}.

\begin{theorem}\label{algebraic-version-4Arf-obstruction}
Let $(C, \phi)$ be the $2n$-dimensional $(-1)^n$-symmetric Poincar\'e complex, and let $(A, \alpha, U)$ be a $(\bb{Z}, m)$-symmetric representation. We shall write $(D, \Gamma) = (C, \phi) \otimes \bar{S}^m(A, \alpha, U)$ and $(D', \Gamma')=(C, \phi) \otimes \bar{S}^m(A, \alpha, \epsilon)$. Here $(D, \Gamma)$ and $(D', \Gamma')$ are $(2n+2m)$-dimensional symmetric complexes and $(2n+2m)$ is divisible by $4$. Then
$$
\sigma(D, \Gamma) - \sigma(D', \Gamma') = 4 \mathrm{Arf}\left( L^{\perp}/L , [\Gamma_0 \oplus -\Gamma'_0], \frac{ \left[\mc{P}_2 \oplus-\mc{P}'_2\right] }{2} \right) \in 4\zz_2 \subset \zz_8,
$$
where $L^{\perp} =\left\{(x,x') \in H^{2k}(D;\zz_2) \oplus H^{2k}(D';\zz_2) \vert \Gamma_0(x, x) =\Gamma'_0(x', x') \in \zz_2 \right\}$ and the Wu sublagrangian $L = \langle (v_{2k}, v'_{2k}) \rangle \subset L^{\perp}$, with $(v_{2k}, v'_{2k})$ the algebraic Wu class $(v_{2k}, v'_{2k}) \in H^{2k}(D;\zz_2) \oplus H^{2k}(D';\zz_2).$
\end{theorem}

Note that when both $m$ and $n$ are odd in the fibration $F^{2m} \to E^{4k} \to B^{2n}$ then by Meyer \cite{Meyerpaper} and by Hambleton, Korzeniewski and Ranicki \cite{modfour} we have that
$$\sigma(E) = 0 \in \zz_4.$$
So  the general formula for the signature mod $8$ of a fibration given geometrically in
\ref{4Arf-general-fibration} and algebraically in \ref{algebraic-version-4Arf-obstruction} simplifies in the case of a fibration $F^{4i+2} \to E^{4k} \to B^{4j+2}$ to the expression in Proposition \ref{4Arf-odd-m-n-geometric}  geometrically or \ref{algebraic-version-4Arf-obstruction-4i+2} algebraically.

\begin{proposition}\label{4Arf-odd-m-n-geometric}
Let $F^{4i+2} \to E^{4k} \to B^{4j+2}$ be an oriented Poincar\'e duality fibration, then
$$\sigma(E) =\mathrm{BK}(H^{2k}(E;\zz_2), \lambda, \mc{P}_2) = 4 \mathrm{Arf}\left( L^{\perp}/L , [\lambda],  \frac{\left[\mc{P}_2\right]}{2}  \right) \in \zz_8,$$
where $L^{\perp} =\left\{x \in H^{2k}(E;\zz_2) \vert \lambda (x,x)= 0 \in \zz_2 \right\}$ and $L = \langle v_{2k} \rangle \subset L^{\perp}$, with $v_{2k}(E)$ the  Wu class $v_{2k}(E) \in H^{2k}(E;\zz_2).$
\end{proposition}

\begin{proof}
Here for dimension reasons $\sigma(F)$ and $\sigma(B)$ are both $0.$ Thus, by \cite[Theorem A]{modfour}, we know that the signature of $E$ is divisible by $4$.
So that we can write
\begin{align*}
\sigma(E) - \sigma(B \times F) &= \sigma(E)        =\mathrm{BK}(H^{2k}(E;\zz_2), \lambda, \mc{P}_2)  \in \zz_8, 
\end{align*}
and since $\sigma(E)= \mathrm{BK}(H^{2k}(E;\zz_2), \lambda, \mc{P}_2) =0 \in \zz_4$, then the result follows as an application of Theorem \ref{4Arf-topology}.
\end{proof}

Algebraically Proposition \ref{4Arf-odd-m-n-geometric} is restated as follows.

\begin{proposition}\label{algebraic-version-4Arf-obstruction-4i+2}
Let $(C, \phi)$ be a $4i+2$-dimensional $(-1)$-symmetric Poincar\'e complex, and let $(A, \alpha, U)$ be a $(\bb{Z}, 2j+1)$-symmetric representation. We shall write $(D, \Gamma) = (C, \phi) \otimes \bar{S}^{2j+1}(A, \alpha, U)$ and $(D', \Gamma')=(C, \phi) \otimes \bar{S}^{2j+1}(A, \alpha, \epsilon)$,
then $\sigma(D', \Gamma')=0 \in \zz$, $\sigma(D, \Gamma)=0 \in \zz_4$ and
$$
\sigma(D, \Gamma) = 4 \mathrm{Arf}\left( L^{\perp}/L , [\Gamma_0],  \frac{\left[\mc{P}_2\right]}{2}  \right) \in 4\zz_2 \subset \zz_8,
$$
where $L^{\perp} =\left\{x \in H^{2k}(D;\zz_2) \vert \Gamma_0(x, x)= 0  \in \zz_2 \right\}$ and $L = \langle v_{2k} \rangle \subset L^{\perp}$, with $v_{2k}$ the algebraic Wu class $v_{2k} \in H^{2k}(D;\zz_2).$
\end{proposition}

\subsubsection{Relation to other expressions in the literature for the signature of a fibre bundle}\label{Other-expressions}
\textbf{The Arf invariant and the second Stiefel-Whitney class}

In \cite{Meyerpaper} Meyer studied the signature of a surface bundle $F^2 \to E^4 \to B^2$, where both $F$ and $B$ are orientable surfaces of genus $h$ and $g$ respectively. Meyer expressed the signature of the total space in terms of the first Chern class of the complex vector bundle $\beta \co B \to BU(h)$ associated to the local coefficient system $\widetilde{B} \times_{\pi_1(B)} \bb{R}^{2h},$ determined by $H_1(- ; \mathbb{R})$ of the fibres,
$$\sigma(E) = 4 c_1(\beta) \in \zz.$$
\begin{gather*}
\tag*{\text{So that}}
\frac{\sigma(E)}{4} = c_1(\beta) \in \zz.
\end{gather*}

From \cite[problem 14-B]{Mil-Sta} the Chern classes of an $h$-dimensional complex vector bundle $\beta \co B \to BU(h)$ are integral lifts of the Stiefel-Whitney classes of the underlying oriented $2h$-dimensional real vector bundle $\beta^{\bb{R}}\co B \to BSO(2h)$. That is, the mod $2$ reduction of the first Chern class is the second Stiefel-Whitney class.
Hence the second Stiefel-Whitney class of the real vector bundle $\beta^{\bb{R}} \co B \to BSO(2h)$ associated to the local coefficient system $\widetilde{B} \times_{\pi_1(B)} \bb{R}^{2h}$ can be expressed as
\begin{equation}\label{sign-stiefel-surf}\frac{\sigma(E)}{4} = w_2(\beta^{\bb{R}}) \in \zz_2.\end{equation}

In Proposition \ref{4Arf-odd-m-n-geometric} we have shown that for a fibration $F^{4i+2} \to E^{4k} \to B^{4j+2}$, the signature mod $8$  can be expressed in terms of the Arf invariant.
In particular for a surface bundle,
\begin{equation}\label{sign-Arf-surf}\sigma(E)  = 4 \mathrm{Arf}\left( L^{\perp}/L , [\lambda],  \frac{\left[\mc{P}_2\right]}{2}  \right) \in \zz_8.\end{equation}

Combining both expressions in \eqref{sign-stiefel-surf} and \eqref{sign-Arf-surf}

$$\mathrm{Arf}\left( L^{\perp}/L , [\lambda],  \frac{\left[\mc{P}_2\right]}{2}  \right) = w_2(\beta^{\bb{R}}) \in \zz_2,$$
with $L^{\perp} =\left\{x \in H^{2k}(E;\zz_2) \vert \lambda(x,x)= 0 \in \zz_2 \right\}$ and $L = \langle v_{2k} \rangle \subset L^{\perp}.$

\medskip

\textbf{The Arf invariant and the Todd genus}

In \cite{Atiyah-cov} Atiyah considers a $4$-manifold $E$ which arises as a complex algebraic surface with a holomorphic projection $\pi \co E \to B$ for some complex structure on $B$. The fibres $F_b = \pi^{-1}(b)$ are algebraic curves but the complex structure varies with $b$.
Atiyah establishes that the Todd genera of $E$, $B$ and $F$ are related to the signature of $E$ by the equation
$$ \frac{\sigma(E)}{4} = T(E)-T(B)T(F) \in \zz.$$
Following this equation the signature modulo $8$ of the total space is detected by the reduction mod $2$ of difference $T(E)-T(B)T(F)$, since $\sigma(E) = 4(T(E)-T(B)T(F)) \in \zz_8$ is equivalent to $ \frac{\sigma(E)}{4} = T(E)-T(B)T(F) \in \zz_2.$
Comparing this with Proposition \ref{4Arf-odd-m-n-geometric} we have that
$$\mathrm{Arf}\left( L^{\perp}/L , [\lambda],  \frac{\left[\mc{P}_2\right]}{2}  \right) = T(E)-T(B)T(F) \in \zz_2,$$
with $L^{\perp} =\left\{x \in H^{2k}(E;\zz_2) \vert \lambda(x,x)= 0 \in \zz_2 \right\}$ and $L = \langle v_{2k} \rangle \subset L^{\perp}.$

\subsection{Multiplicativity mod $8$ in the $\zz_4$-trivial case.} \label{multiplicative-mod8}

In the previous section we have shown that in general there is an obstruction to multiplicativity modulo $8$ of a fibration  $F^{2m} \to E^{4k} \to B^{2n}$ given by the Arf invariant.

We shall now prove that by imposing the condition that  $\pi_1(B)$ acts trivially on $H^{m}(F, \bb{Z})/torsion \otimes \bb{Z}_4$, the obstruction disappears and the fibration has signature multiplicative modulo $8$. Geometrically the theorem is stated as follows.

\begin{theorem}\label{mod8-theorem} Let $F^{2m} \to E \to B^{2n}$ be an oriented Poincar\'e duality fibration. If the action of $\pi_1(B)$ on $H^m(F, \bb{Z})/torsion \otimes \bb{Z}_4$ is trivial, then
$$\sigma(E) - \sigma(F) \sigma(B) = 0\in \zz_8.$$
\end{theorem}

\begin{rem}
Clearly by the definition of the signature, when $m$ and $n$ are odd $$\sigma(F)\sigma(B)=0 \in \zz.$$
So in this case, the theorem establishes that
$$\sigma(E) \equiv 0 \in \zz_8.$$
\end{rem}

\subsubsection{Tools for proving Theorem \ref{mod8-theorem}}

We shall prove Theorem \ref{mod8-theorem} by proving its algebraic analogue, which we  state as Theorem \ref{mod8-algebraic}. In this section we prove some results that will be needed in the proof of the algebraic statement of the theorem (as in Theorem \ref{mod8-algebraic}).

The \textbf{algebraic analogue} of the condition that  $\pi_1(B)$ acts trivially on $H^{m}(F, \bb{Z})/torsion \otimes \bb{Z}_4$ is defined as follows:
\begin{definition}\label{Z4-triviality}
A $(\bb{Z}, m)$-symmetric representation $(A, \alpha, U)$ of a group ring $\bb{Z}[\pi]$ \textbf{is $\bb{Z}_4$-trivial} if
$$U(r) \otimes 1= \epsilon(r)\otimes1 \co A \otimes \bb{Z}_4 \to A \otimes \bb{Z}_4$$
for all $r\in \bb{Z}[\pi],$ where $\epsilon$ denotes the trivial action homomorphism $$\epsilon \co \zz[\pi_1(B)] \to H_0(\textnormal{Hom}_{\zz}(A, A))^{op}.$$
\end{definition}

A weaker condition is that of \textbf{$\zz_2$-triviality}:
\begin{definition}{(Korzeniewski \cite[chapter 8]{Korzen})} \qua \label{Z2-triviality}
A $(\bb{Z}, m)$-symmetric representation $(A, \alpha, U)$ of a group ring $\bb{Z}[\pi]$ \textbf{is $\bb{Z}_2$-trivial} if
$$U(r) \otimes 1= \epsilon(r)\otimes1 \co A \otimes \bb{Z}_2 \to A \otimes \bb{Z}_2$$
for all $r\in \bb{Z}[\pi],$ where $\epsilon$ denotes the trivial action homomorphism as in the previous definition.
\end{definition}

\begin{rem} In particular any statement which holds under assumption of a $\zz_2$-trivial action is also true for an assumption of $\zz_4$-triviality. The converse may not be true.
\end{rem}

In the statement of the algebraic analogue of Theorem \ref{mod8-theorem}  we shall let $(C, \phi)$ be a $2n$-dimensional $(-1)^n$-symmetric Poincar\'e complex over $\zz[\pi]$ and $(A, \alpha, U)$ be a $(\zz, m)$-symmetric representation with $\zz_4$-trivial $U \co \zz[\pi] \to H_0 (\textnormal{Hom}_{\zz}(A, A))^{op}$. When a result is true with a $\zz_2$-trivial action we shall indicate this accordingly.

With $4k =2m+2n$, we shall write $(D, \Gamma)$ for the $4k$-dimensional symmetric Poincar\'e complex over $\zz$ given by
$$(D, \Gamma) = (C, \phi) \otimes_{\zz[\pi]} \bar{S}^m(A, \alpha, U),$$
and $(D', \Gamma')$ for the $4k$-dimensional symmetric Poincar\'e complex over $\zz$ given by
$$ (D', \Gamma') = (C, \phi) \otimes_{\zz[\pi]}\bar{S}^m(A, \alpha, \epsilon),$$
where $(A, \alpha, \epsilon)$ is the trivial representation.

\begin{lemma} \label{lemma-mod8-algebraic}
If the representation $(A, \alpha, U)$ is $\zz_2$-trivial then $$\zz_2 \otimes_{\zz}(D, \Gamma) \cong \zz_2 \otimes_{\zz}(D', \Gamma').$$ That is,
\begin{itemize}
\item[(i)] $H^{2k}(D, \zz_2) \cong H^{2k}(D', \zz_2).$
\item[(ii)] the symmetric structure reduced mod $2$ is the same for both symmetric complexes $(D, \Gamma)$ and $(D', \Gamma').$
\end{itemize}
\end{lemma}

\begin{proof} \label{proof-iso} If $(A, \alpha, U)$ is a $\zz_2$-trivial representation then
$$
\begin{array}{ccc}
\bb{Z}_2 \otimes_{\bb{Z}} (D, \Gamma) & = ~ \zz_2 \otimes (C, \phi) \otimes_{\zz[\pi]} \bar{S}^m(A, \alpha, U) & \\
&\cong ~ (C, \phi) \otimes_{\zz[\pi]} \bar{S}^m(A, \alpha, U) \otimes \zz_2 & \\
&\cong ~ (C, \phi) \otimes_{\zz[\pi]} \bar{S}^m(A, \alpha, \epsilon) \otimes \zz_2 & \\
&\cong ~ \zz_2 \otimes (C, \phi) \otimes_{\zz[\pi]} \bar{S}^m(A, \alpha, \epsilon) & =~ \bb{Z}_2 \otimes_{\bb{Z}} (D', \Gamma'),
\end{array}
$$
which proves the Lemma and clearly implies (i) and (ii).
\end{proof}

From the proof of Lemma \ref{lemma-mod8-algebraic}, we know that for a $\bb{Z}_2$-trivial twisted product and for an untwisted product, the vector spaces given by the middle dimensional cohomology with $\bb{Z}_2$ coefficients  are isomorphic and the $\bb{Z}_2$-valued symmetric structure is also the same in both cases. Clearly the results in this lemma also hold when the action $U$ is $\zz_4$-trivial.

\subsubsection{Comparing Pontryagin squares for the twisted and untwisted product}\label{comparing-pont}
In section \ref{alg-pont}
 we defined algebraic Pontryagin squares depending on the symmetric structure of a symmetric complex.
A $\zz_4$-valued quadratic function such as the Pontryagin square cannot be recovered uniquely from the associated $\bb{Z}_2$-valued bilinear pairing $\Gamma_0 \co H^{2k}(D;\zz_2)\times H^{2k}(D;\zz_2) \to \bb{Z}_2$. It is crucial for the proof of Theorem \ref{mod8-theorem}  to note that the Pontryagin square depends on the definition of the symmetric structure on integral cochains.
In particular, suppose that
$$(D, \Gamma) = (C, \phi) \otimes_{\zz[\pi]} \bar{S}^m(A, \alpha, U),$$
$$ (D', \Gamma') = (C, \phi) \otimes_{\zz[\pi]}\bar{S}^m(A, \alpha, \epsilon),$$
where $(A, \alpha, U)$ is a $\mathbb{Z}_2$-trivial representation.


Then depending on the integral symmetric structures, we have two different lifts of the modulo $2$ symmetric structure on
$$V = H^{2k}( D; \mathbb{Z}_2) = H^{2k}(D'; \mathbb{Z}_2)$$
which give rise to two different Pontryagin squares.

In other words we are considering two chain complexes which are different over $\zz$, $(D, \Gamma)$ and $(D', \Gamma')$ but are chain equivalent when reduced over $\zz_2$
$$(D, \Gamma) \otimes_{\zz}\zz_2 \cong (D', \Gamma')\otimes_{\zz}\zz_2.$$
But the integral symmetric structures $ \Gamma$ and $ \Gamma'$  are different in the twisted and untwisted products, so this gives rise in general to two different Pontryagin squares of the same $\bb{Z}_2$-valued symmetric bilinear form.

\begin{rem} In what follows  we will use the notation $\mc{P}_{2}$ for the twisted Pontryagin square and $\mc{P}'_2$ for the Pontryagin square on an untwisted product.
\end{rem}

The following Proposition \ref{differ-linear-map} applies precisely to the general situation of the two Pontryagin squares $\mc{P}_2$ and $\mc{P}'_2$ that we have just described above.

\begin{proposition} \label{differ-linear-map}
Let $V$ be a $\bb{Z}_2$-valued vector space and $\lambda \co V \otimes V\to \bb{Z}_2$ a non-singular symmetric bilinear pairing. Any two quadratic  enhancements $q, q' \co V \to \bb{Z}_4$ over $\lambda$ differ by a linear map
$$q'(x)-q(x) = 2 \lambda(x, t) \in \zz_4$$
for some $t \in V.$
\end{proposition}
\begin{proof}
See \cite{Taylor}, \cite{Deloup}, \cite[page 10]{Mod8}.
\end{proof}

Consequently from this proposition we know that the two Pontryagin squares $\mc{P}_2$ and $\mc{P}'_2$ differ by linear map.
Furthermore the Brown-Kervaire invariants of two quadratic enhancements as in Proposition \ref{differ-linear-map} are related by the following theorem.
\begin{theorem}{\rm (Brown \cite[Theorem 1.20 (x)]{Brown})} \qua
Let $V$ be a $\bb{Z}_2$-valued vector space and $\lambda \co V \otimes V\to \bb{Z}_2$ a non-singular symmetric bilinear pairing, then any two quadratic  enhancements $q, q' \co V \to \bb{Z}_4$ over $\lambda$ differ by a linear map, $q'(x)-q(x) = 2 \lambda(x, t) \in \zz_4$ and
$$\mathrm{BK}(V,\lambda,q)-\mathrm{BK}(V, \lambda,q')=2q(t) \in \zz_8$$
for some $t \in V$.
\end{theorem}
Note that when $\mathrm{BK}(V,\lambda,q)$ and $\mathrm{BK}(V,\lambda,q')$ are divisible by $4$ we can write this relation in  terms of the Arf invariant
$$\mathrm{Arf}(W, \mu ,h)-\mathrm{Arf}(W, \mu, h')=h(t) \in \zz_2.$$

We will now show that by setting the condition of a $\zz_4$-trivial action $U$ there is an isomorphism between the untwisted Pontryagin square and the $\zz_4$-twisted Pontryagin square, in which case $h(t)=0.$

\subsubsection{Pontryagin squares and $\mathbb{Z}_4$-trivial representations}\label{Pont-action}

In the previous subsection we emphasized that a $\mathbb{Z}_4$-valued quadratic function such as the Pontryagin squares cannot be recovered uniquely from the associated $\mathbb{Z}_2$-valued bilinear pairing.

Here $(D, \Gamma)= (C, \phi) \otimes_{\zz[\pi]} \bar{S}^m(A, \alpha, U)$ is as before a $4k$-dimensional symmetric Poincar\'e complex over $\zz$, with $U$ the action of the $\zz[\pi]$ on $(A, \alpha)$, and $(D', \Gamma')= (C, \phi) \otimes_{\zz[\pi]} \bar{S}^m(A, \alpha, \epsilon)$ the $4k$-dimensional symmetric Poincar\'e complex  with $(A, \alpha,  \epsilon)$  the trivial representation.

We will now show that the Pontryagin squares on $(D, \Gamma)$ and $(D', \Gamma')$  can be described in terms of $\mathbb{Z}_4 \otimes_{\mathbb{Z}}(D, \Gamma)$ and $\mathbb{Z}_4 \otimes_{\mathbb{Z}} (D', \Gamma')$ respectively.

Firstly we note the following Lemma, which is similar to Lemma \ref{lemma-mod8-algebraic}.
\begin{lemma} \label{lemma-Z4-trivial}
If the representation $(A, \alpha, U)$ is $\zz_4$-trivial then $$\zz_4 \otimes_{\zz}(D, \Gamma) \cong \zz_4 \otimes_{\zz}(D', \Gamma').$$
\end{lemma}
\begin{proof} \label{proof-iso-Z4} The proof follows from the same argument as the proof of Lemma \ref{lemma-mod8-algebraic}.
\end{proof}
\begin{lemma}\label{Z4-pont} If the action $U$ is $\zz_4$-trivial then the $\zz_4$-twisted Pontryagin square on $(D, \Gamma)$ is isomorphic to the untwisted Pontryagin square on $(D', \Gamma').$
\end{lemma}
\begin{proof}
 Let $B(2k, 2)$ be as before the $\mathbb{Z}$-module chain complex concentrated in dimensions $2k+1$ and $2k$,
$$B(2k, 2) \co \dots \rightarrow \mathbb{Z} \xrightarrow{2} \mathbb{Z} \rightarrow,$$
and let $\bar{B}(2k, 2)$ be a shifted resolution of $\mathbb{Z}_2$ as a module over $\mathbb{Z}_4,$
$$\dots \leftarrow 0 \leftarrow 0 \leftarrow \mathbb{Z}_4 \xleftarrow{.2}  \mathbb{Z}_4    \xleftarrow{.2}   \mathbb{Z}_4     \xleftarrow{.2}  \dots$$
so that $H_{2k}(\bar{B}(2k, 2)) \cong \mathbb{Z}_2$ and all other homology are $0$.
There is a chain map unique up to chain homotopy from $B(2k, 2)$ to $\bar{B}(2k, 2)$, which induces an isomorphism in homology.
This chain map also induces an injective map in the symmetric $Q$-groups
$$Q^{4k}(B(2k, 2)) = \mathbb{Z}_4 \longrightarrow Q^{4k}(\bar{B}(2k, 2)).$$
An element $x \in H^{2k}(D; \mathbb{Z}_2)$ can be represented by a chain map
$$f_x \co \mathbb{Z}_4 \otimes_{\mathbb{Z}} D \longrightarrow \bar{B}(2k, 2),$$
which is unique up to chain homotopy. The map $f_x$ induces a homomorphism of symmetric $Q$-groups
$$(f_x)^{\%}  \co Q^{4k}(\mathbb{Z}_4 \otimes_{\mathbb{Z}} D) \longrightarrow Q^{4k}(\bar{B}(2k, 2)).$$
Evaluating this homomorphism on $\mathbb{Z}_4 \otimes \Gamma \in Q^{4k}(\mathbb{Z}_4 \otimes_{\mathbb{Z}} D )$ we get the Pontryagin square as an element $\mathcal{P}_2(x) \in Q^{4k}(\bar{B}(2k, 2)).$ Hence the algebraic Pontryagin square on $H^{2k}(D; \mathbb{Z}_2)$ can be expressed in terms of $\mathbb{Z}_4 \otimes (D, \Gamma)$ only.

A similar argument holds for the Pontryagin square on $(D', \Gamma').$

 By Lemma \ref{lemma-Z4-trivial} we know that if $(A, \alpha, U)$ is a $\zz_4$-trivial representation then $$\zz_4 \otimes_{\zz}(D, \Gamma) \cong \zz_4 \otimes_{\zz}(D', \Gamma'),$$ and since the Pontryagin squares $\mathcal{P}_2$ and $\mathcal{P}'_2$ on $(D, \Gamma)$ and $(D', \Gamma')$   only depend on  $\zz_4 \otimes_{\zz}(D, \Gamma)$ and on $\zz_4 \otimes_{\zz}(D', \Gamma')$ respectively, then we deduce that these Pontryagin squares  are isomorphic.
\end{proof}

\subsubsection{The algebraic analogue of Theorem \ref{mod8-theorem}}

We can now state and prove the algebraic analogue of Theorem \ref{mod8-theorem}. The proof of the algebraic analogue implies the proof of the geometric statement.
\begin{theorem} \label{mod8-algebraic} Let $(D, \Gamma)$ be a $4k$-dimensional symmetric Poincar\'e complex over $\zz$ of the form
$$(D, \Gamma)=  (C, \phi) \otimes \bar{S}^m(A, \alpha, U),$$
with $U$ the action of $\zz[\pi]$ on $(A, \alpha).$

If the representation $(A, \alpha, U)$ is $\zz_4$-trivial then
$$\sigma(D, \Gamma) - \sigma(D', \Gamma') = 0 \in \zz_8,$$
where $(D', \Gamma')=  (C, \phi) \otimes \bar{S}^m(A, \alpha, \epsilon)$ is the trivial product.
\end{theorem}

\begin{proof}
The signatures modulo $8$ of both $(D, \Gamma)$ and $(D', \Gamma')$ are given by Morita's theorem \ref{Morita-theorem-cx} to be
\begin{itemize}
\item $\sigma(D, \Gamma) \equiv \mathrm{BK}(H^{2k}(D; \zz_2), \Gamma_0, \mc{P}_2) \pmod{8}$
\item $\sigma(D', \Gamma') \equiv \mathrm{BK}(H^{2k}(D'; \zz_2), \Gamma_0, \mc{P}'_2) \pmod{8}.$
\end{itemize}

From Lemma \ref{lemma-mod8-algebraic} we know that
$$(H^{2k}(D; \zz_2), \Gamma_0) \cong (H^{2k}(D'; \zz_2), \Gamma'_0),$$
and from Lemma \ref{Z4-pont} we know that the two Pontryagin squares $\mc{P}_2$ and $\mc{P}'_2$ are isomorphic, hence
\begin{align*}
\sigma(D, \Gamma) & = \mathrm{BK}(H^{2k}(D; \zz_2), \Gamma_0, \mc{P}_2) \\
                            & =  \mathrm{BK}(H^{2k}(D'; \zz_2), \Gamma'_0, \mc{P}'_2) \\
                             & = \sigma(D', \Gamma') \in \zz_8
\end{align*}
so the result follows. \proved
\end{proof}

The geometric proof of Theorem \ref{mod8-theorem}, which states that $\sigma(E)- \sigma(B \times F) = 0 \in \zz_8$ for a fibration $F^{2m}\to E^{4k} \to B^{2n}$ with trivial action of $\pi_1(B)$ on $H^m(F, \zz)/torsion \otimes \zz_4$  is a consequence of the algebraic proof:

Let $F^{2m} \to E^{4k} \to B^{2n}$ be Poincar\'e fibration with $\zz_4$-trivial action. Then according to Remark \ref{equal-signatures},  the
$\zz_4$-enhanced symmetric forms over $\zz_2$
$$(H^{2k}(E;\zz_2),\lambda_E,q_E)~,~
(H^{n}(B;\zz_2)\otimes_{\zz_2}H^{m}(F;\zz_2),\lambda_B\otimes \lambda_F,q_B \otimes q_F)$$
of the symmetric Poincar\'e complexes over $\zz_4$
$$(C(E;\zz_4),\phi_E)~,~(C(B;\zz_4),\phi_B) \otimes (C(F;\zz_4),\phi_F)$$
have the same signature, so that
\begin{align*}
\sigma(E)&= \mathrm{BK}(H^{2k}(E;\zz_2),\lambda_E,q_E) \\[1ex]
&=\mathrm{BK}(H^{2k}(B \times F;\zz_2),\lambda_B,q_B) \\[1ex]
&=\sigma(B \times F) \in \zz_8 \\[1ex]
& = \sigma(B) \sigma(F) \in \zz_8.
\end{align*}

For Theorem \ref{mod8-theorem} to be true with a $\zz_2$-trivial action, we would need to prove that there exists an isomorphism of the untwisted and the twisted Pontryagin squares with $U$ a $\zz_2$-trivial action. At the moment it is only clear that two such Pontryagin squares differ by a linear map as explained in Proposition \ref{differ-linear-map}. However there is no problem if the action is $\zz_4$-trivial, as shown in Theorem \ref{mod8-theorem}.
\newpage

\section*{Appendix} \label{Appendix}

In this appendix we summarize the notation used throughout the paper.
\begin{center}
  \begin{tabular}{ l  l }
 \\ \textbf{In Algebra:}  \\   \\ \hline
      $(C, \phi)$ & A $2n$-dimensional  $(-1)^n$-symmetric Poincar\'e \\  & complex over $\zz[\pi]$. (\cite{atsI}) \\ \hline

  $(A, \alpha)$  & A nonsingular $(-1)^m$-symmetric form over $\zz$  \\ & with $\alpha \co A \to A^*$ (Definition \ref{representation A})  \\ \hline
  $(A, \alpha, U)$ & A $(\zz, m)$-symmetric representation with \\ & $U \co \bb{Z}[\pi] \to H_0( \textnormal{Hom}_{\zz}(A, A))^{op}$ \\ & (Definition \ref{functor-homology}) \\ \hline
$\bar{S}^m(A, \alpha, U)$ & Skew-suspension of $(A, \alpha, U)$ \\ \hline
$(C, \phi) \otimes \bar{S}^m(A, \alpha, U) = (D, \Gamma)$ {\color{White}aaaaaaaaa}& A $4k=2m+2n$-dimensional symmetric  \\  & complex over $\zz$ with the action of $\pi$\\ &  given by $U$. (Proposition \ref{two-functors}). \\ \hline
  $(C, \phi) \otimes \bar{S}^m(A, \alpha, \epsilon) = (D', \Gamma')$  & The $4k$-dimensional symmetric complex  \\  & over $\zz$ with the action of $\pi$ given by \\ & the trivial action $\epsilon$.  (Proposition \ref{two-functors})  \\ \hline

 $C^{(2n)-r+s} \otimes  \bar{S}^mA^* \xrightarrow{U(\phi_s)  (\alpha)}C_r \otimes  \bar{S}^mA $ {\color{White}{aaaaaaa}}& The symmetric structure of a twisted product.  \\ & (Theorem \ref{chain-iso} and \cite{SurTransfer}) \\ \hline
  $\mc{P}_2$  & Pontryagin square depending on  \\ & the symmetric  structure  $\Gamma$ (Section \ref{alg-pont}) \\ \hline
 $\mc{P}'_2$ & Pontryagin square depending on  \\ & the symmetric  structure $\Gamma'$ (Section \ref{alg-pont}) \\ \hline

 \\ \textbf{In Topology:}  \\   \\ \hline
      $(C(\tilde{B}), \phi)$ & The $2m$-dimensional symmetric  complex \\ &  over $\zz[\pi_1(B)]$ of the universal cover \\  &of the base. (\cite{atsI}) \\ \hline
  $A = H^{m}(F, \bb{Z})$    & The middle dimensional cohomology \\& of the fibre $F^{2m}$ (Definition \ref{representation A}) \\ \hline
\end{tabular}
\end{center}


\begin{center}
  \begin{tabular}{ l  l }
  $(A, \alpha)$  & The nonsingular $(-1)^m$-symmetric form   \\ & with $\alpha= (-1)^m \alpha^* \co A \to A^*$ \\ & (Definition \ref{representation A})  \\ \hline
  $(A, \alpha, U)$ & The symmetric representation with \\ & $U \co \bb{Z}[\pi_1(B)] \to H_0( \textnormal{Hom}_{\zz}(A, A))^{op}$ \\ & (Definition \ref{functor-homology}) \\ \hline
$\bar{S}^m(A, \alpha, U)$ & $2m$-dimensional $(-1)^m$-symmetric complex \\& given by skew-suspension of $(A, \alpha, U)$ \\ &  (See Remark \ref{skew-susp}) \\ \hline

$(C(\tilde{B}), \phi) \otimes_{\zz[\pi_1(B)]} \bar{S}^m(A, \alpha, U)$ {\color{White}aaaaaaaaa}& The chain complex model for the \\  & total space with the action of $\pi_1(B)$\\ &  given by $U$. (Proposition \ref{two-functors})\\ \hline
$\sigma \left((C(\tilde{B}), \phi) \otimes_{\zz[\pi_1(B)]} \bar{S}^m(A, \alpha, U) \right)$  {\color{White}aaaaaaaaa}& The signature of the total space \\ $= \sigma(E) \in \zz$& is the signature of the chain complex model \\  & for the total space with the action of $\pi_1(B)$\\ &  given by $U$. (\cite[Theorem 4.9]{Korzen})\\ \hline

  $(C(\tilde{B}), \phi) \otimes_{\zz[\pi_1(B)]} \bar{S}(A, \alpha, \epsilon)$  & The chain complex model for the  \\  & total space with trivial  action of $\pi_1(B)$, \\ & which we denote by $\epsilon$.  (Proposition \ref{two-functors}) \\ \hline

$\sigma \left((C(\tilde{B}), \phi) \otimes_{\zz[\pi_1(B)]} \bar{S}^m(A, \alpha, \epsilon) \right)$  {\color{White}aaaaaaaaa}& The signature of the trivial product $B \times F$ \\ $=\sigma(C(\tilde{B}), \phi) \otimes_{\zz[\pi_1(B)]} \sigma(\bar{S}^m(A, \alpha, \epsilon)) $& is the signature of the chain complex model  \\  $= \sigma(B) \sigma(F) \in \zz$& for the total space with the action of $\pi_1(B)$\\ &  given by $\epsilon$. (\cite[Theorem 4.9]{Korzen})\\ \hline

 $C^{(2n)-r+s}(\tilde{B}) \otimes  \bar{S}A^* \xrightarrow{U(\phi_s)  (\alpha)}C_r(\tilde{B})\otimes  \bar{S}A $ & symmetric structure of a twisted product.  \\ & (Theorem \ref{chain-iso} and \cite{SurTransfer}) \\ \hline
  $\mc{P}_2$   &Pontryagin square of twisted product \\ &depending on the symmetric  structure \\ &  of  $(C(\tilde{B}), \phi) \otimes \bar{S}(A, \alpha, U)$ (Section \ref{alg-pont}) \\ \hline

 $\mc{P}'_2$ &Pontryagin square of untwisted product \\ &depending on the symmetric  structure \\ &  of  $(C(\tilde{B}), \phi) \otimes \bar{S}(A, \alpha, \epsilon)$ (Section \ref{alg-pont}) \\ \hline

  \end{tabular}
\end{center}


%
%

%
\bibliographystyle{gtart}

\bibliography{biblio-TT}



\end{document}